\documentclass[9pt,shortpaper,twoside,web]{ieeecolor}
\usepackage{generic}
\usepackage{cite}
\usepackage{amsmath,amssymb,amsfonts}
\usepackage{algorithmic}
\usepackage{graphicx}
\usepackage{textcomp}
\def\BibTeX{{\rm B\kern-.05em{\sc i\kern-.025em b}\kern-.08em
    T\kern-.1667em\lower.7ex\hbox{E}\kern-.125emX}}
\usepackage{amsthm}
\usepackage{paralist}

\newcommand{\bb}[1]{\mathbb{ #1 }}
\newcommand{\N}{\bb{N}}
\newcommand{\R}{\bb{R}}
\newcommand{\Z}{\bb{Z}}

\DeclareMathOperator*{\argmax}{arg\,max}

\newtheorem{theorem}{Theorem}
\newtheorem{lemma}{Lemma}

\theoremstyle{definition}
\newtheorem{assumption}{Assumption}
\newtheorem{definition}{Definition}
\newtheorem{example}{Example}
\newtheorem{remark}{Remark}

\begin{document}
\title{Compactly Restrictable Metric Policy Optimization Problems}
\author{Victor D. Dorobantu, Kamyar Azizzadenesheli, and Yisong Yue
\thanks{Submitted May 15th, 2021. Resubmitted July 6th, 2022.
This work was supported in part by DARPA and Beyond Limits. Victor D. Dorobantu was also supported in part by a Kortschak Fellowship.}
\thanks{Victor D. Dorobantu and Yisong Yue are with the Department of Computing and Mathematical Sciences, California Institute of Technology, Pasadena, CA 91125 USA. (e-mails: \{vdoroban, yyue\}@caltech.edu). Yisong Yue is affiliated with Argo AI, Pittsburgh, PA. Kamyar Azizzadenesheli is with 
the Department of Computer Science, Purdue University, West Lafayette, IN 47907 USA (e-mail: kamyar@purdue.edu).}
}

\maketitle

\begin{abstract}
We study policy optimization problems for deterministic Markov decision processes (MDPs) with metric state and action spaces, which we refer to as Metric Policy Optimization Problems (MPOPs). Our goal is to establish theoretical results on the well-posedness of MPOPs that can characterize practically relevant continuous control systems. To do so, we define a special class of MPOPs called Compactly Restrictable MPOPs (CR-MPOPs), which are flexible enough to capture the complex behavior of robotic systems but specific enough to admit solutions using dynamic programming methods such as value iteration. We show how to arrive at CR-MPOPs using forward-invariance. We further show that our theoretical results on CR-MPOPs can be used to characterize feedback linearizable control affine systems.
\end{abstract}

\begin{IEEEkeywords}
Continuous Markov Decision Processes, Reinforcement Learning, Optimal Control, Value Iteration, Selection Theorems, Sampled-Data, Physical Systems
\end{IEEEkeywords}

\section{Introduction}
\label{sec:introduction}

\IEEEPARstart{P}{olicy} optimization is a cornerstone in planning, control and reinforcement learning. Classic approaches include value iteration and policy iteration \cite{dynkin1979controlled, puterman2014markov, bertsekas2011dynamic, bertsekas2019reinforcement, sutton2018reinforcement}. From a theoretical standpoint, a key step is establishing when policy optimization is well-posed, i.e., when optimal policies exist, and when they can be found algorithmically. While such results for value and policy iteration are well established for discrete systems with finitely many states and actions (also known as the tabular setting), relatively few foundational results have been established for continuous control.

In this paper, we study Metric Policy Optimization Problems (MPOPs), which come endowed with metric state and action spaces. Compared to tabular MDPs, several new challenges arise in the continuous setting. Even for deterministic problems, rewards may be unbounded, maxima of functions (over actions) may not exist, and the domains of value functions may be different for different policies. Without addressing these challenges, optimal policies need not exist and value iteration or policy iteration may be impossible.

In order to establish well-posedness of dynamic programming approaches, we define the class of Compactly Restrictable MPOPs (CR-MPOPs). We show that CR-MPOPs arise naturally when imposing forward-invariance constraints on the policy class one optimizes over. As such, CR-MPOPs are well suited for characterizing many systems which rely on forward-invariance for controller design.

Sampled-data design \cite{nevsic1999sufficient} allows us to synthesize policies for continuous-time systems when inputs are passed through a zero-order hold (held constant over fixed frequency time intervals), as is realistic for many physical systems. We leverage recent results \cite{taylor2021sampled, taylor2022safety} to certify the existence of a compact subset of the state space that can be rendered forward-invariant through control. These results are readily applicable to feedback linearizable control affine systems, allowing us to design CR-MPOPs for a wide array of complex systems.

Many frameworks extend the theory of discrete MDPs to general state and action spaces. Most relevant for our work, \cite{dynkin1979controlled} develops semicontinuous MDPs with Borel measurable policy classes, \cite{blackwell1974optimal} and \cite{bertsekas1996stochastic} develop results for analyticially and universally measurable policy classes, respectively, and \cite{feinberg2012average} refines conditions for the well-posedness of value iteration. For a more complete summary, see \cite{yu2015mixed}. Our work focuses on developing classes of policy optimization problems that can naturally connect theoretical results from reinforcement learning with nonlinear continuous-time control, and can be directly translated to checking certain properties in control systems.

Methods that maintain the forward-invariance of subsets of the state space of interest have been well-studied \cite{blanchini1999set, prajna2005optimization, ames2016control}. These and related methods have recently found applications in safe reinforcement learning \cite{Berkenkamp2017SafeRL, rosolia2017learning, cheng2019end, choi2020reinforcement}. Physical systems and robotic platforms have been popular applications for reinforcement learning methods recently, with the majority of methods employing function approximation, discretizations (such as fixed or adaptive gridding or state aggregation), or direct policy search \cite{kober2013reinforcement,levine2016end, lillicrap2015continuous}. 
While policy search methods generally have no guarantees of global optimality, they have received much attention due to the scalability problems of dynamic programming methods (especially after discretization) and convergence problems of approximate dynamic programming (for a more complete discussion of the relative merits of these methods, see \cite[\S 2.3]{kober2013reinforcement}). In contrast, our goal is to identify settings that are compatible with value iteration and to use control theoretic tools to guide solution methods so we can expect good performance.

\subsection{Contributions}
We develop our results in the following phases:
\begin{inparaenum}[1)]
    \item We describe a generic property (Assumption \ref{as:well-posed}) that admits the well-posedness of value iteration for MPOPs (Theorem \ref{thm:well-posed}).
    \item We leverage the forward-invariance of compact sets to design MPOPs and policy classes that satisfy Assumption \ref{as:well-posed}, allowing us to prove well-posedness of a large class of policy optimization problems (Theorem \ref{thm:invariance}). Such problem settings are called CR-MPOPs.
    \item We apply our results on CR-MPOPs to analyze control affine systems with time-sampled control inputs.
    \item We further apply our results to analyze robotic systems with time-sampled control inputs, which comprise a large class of control affine systems. These results translate the generic requirements for well-posedness to concrete requirements on the robotic system.
    \item We finally show that full-state feedback linearizable control affine systems can render a compact subset of the state space forward invariant with continuous controllers (thus satisfying conditions of Theorem \ref{thm:invariance}), demonstrating practically relevant instances of CR-MPOPs.
\end{inparaenum}

\subsection{Notation, Conventions, and Definitions}
We consider nonempty metric spaces throughout this paper and endow each such space with the $\sigma$-algebra generated by its topology, called the Borel $\sigma$-algebra. In this case, open sets and closed sets are measurable sets and continuous functions are measurable functions. A metric space is separable if it contains a countable dense subset. Finite or countable sets can be regarded as separable metric spaces when equipped with the discrete metric \cite[\S 3.A]{kechris2012classical}. For nonempty metric spaces $X$ and $Y$, the set of measurable functions from $X$ to $Y$ is denoted $\mathcal{L}^0(X; Y)$. A bounded measurable function $f: X \to \R$ is upper semicontinuous if for any $c \in \R$, the preimage $f^{-1}([c, \infty)) = \{ x \in X: f(x) \geq c \}$ is closed.

The set of all subsets of $Y$ is the powerset of $Y$, denoted $\mathcal{P}(Y)$. A set-valued function $C: X \to \mathcal{P}(Y)$ is called a correspondence.
The graph of $C$ is defined as:
\begin{equation}
    \mathrm{Graph}(C) = \{ (x, y) \in X \times Y: y \in C(x) \}.
\end{equation}
If $C(x) \neq \emptyset$ for all $x \in X$, then a selector of $C$ is a function $f: X \to Y$ satisfying $f(x) \in C(x)$ for all $x \in X$. For any set $B \subseteq Y$, the lower preimage of $B$ under $C$ is defined as:
\begin{equation}
    C^{\ell}(B) = \{ x \in X: C(x) \cap B \neq \emptyset \}.
\end{equation}
The correspondence $C$ is upper hemicontinuous\footnote{Some authors call such correspondences upper semicontinuous; we use hemicontinuous to distinguish correspondences from real-valued functions.} if the lower preimage of each closed set is a closed set \cite[Lemma 17.4]{aliprantis2006infinite}.

A nonempty subset $X_0 \subseteq X$ is a metric space when equipped with the restriction of the metric on $X$ to $X_0 \times X_0$. For any open set $U \subseteq X_0$, there is an open set $V \subseteq X$ such that $U = X_0 \cap V$. For any continuous function $f: X \to Y$, the restriction $\left.f\right|_{X_0}: X_0 \to Y$ is continuous. Similarly, for any measurable set $A \subseteq X_0$, there is a measurable set $B \subseteq X$ such that $A = X_0 \cap B$, and for any measurable function $f: X \to Y$, the restriction $\left.f\right|_{X_0}$ is measurable. 

Consider a nonempty metric space $Z$ and a function $f: X \times Y \to Z$. For $x \in X$ and $y \in Y$, the $x$-section $f_x: Y \to Z$ and the $y$-section $f^y: X \to Z$ are defined as:
\begin{align}
    f_x(y') &= f(x, y'), & f^y(x') = f(x', y),
\end{align}
for all $x' \in X$ and $y' \in Y$. If $f$ is continuous, then $f$ has continuous sections, and if $f$ is measurable, then $f$ has measurable sections. We also define sections for functions on graphs. Consider a correspondence $C: X \to \mathcal{P}(Y)$ with $C(x) \neq \emptyset$ for all $x \in X$, and a function $f: \mathrm{Graph}(C) \to Z$. For $x \in X$, we define the $x$-section $f_x: C(x) \to Z$ as above.

$\N$ denotes the natural numbers, $\Z_+$ denotes the nonnegative integers, $\R_+$ and $\R_{++}$ denote nonnegative and positive reals, and $\bb{S}^d_+$ and $\bb{S}^d_{++}$ denote $d \times d$ positive semidefinite and definite matrices.

\section{Metric Policy Optimization Problems}
\label{sec:mdp}
In this paper, a deterministic Markov decision process (MDP) will be characterized by a tuple $(\mathcal{S}, \mathcal{A}, C, f, r, \gamma)$, with state space $\mathcal{S}$, action space $\mathcal{A}$, action-admissibility correspondence $C: \mathcal{S} \to \mathcal{P}(\mathcal{A})$ satisfying $C(s) \neq \emptyset$ for each state $s \in \mathcal{S}$, transition map $f: \mathrm{Graph}(C) \to \mathcal{S}$, reward function $r: \mathrm{Graph}(C) \to \R$, and discount factor $\gamma \in [0, 1)$. Note while $r$ may be unbounded, $r$ cannot assume the values $\pm \infty$. We refer to an MDP as a metric MDP if $\mathcal{S}$ and $\mathcal{A}$ are nonempty separable metric spaces and $f$ and $r$ are measurable functions. In this paper, we focus solely on metric MDPs, and often refer to them simply as MDPs for brevity.

We will limit our consideration to deterministic, Markovian, and stationary (time-invariant) policies. In this case, a policy $\pi: \mathcal{S} \to \mathcal{A}$ is a selector of $C$; that is, $\pi(s) \in C(s)$ for all states $s \in \mathcal{S}$. The corresponding closed-loop transition map $f_\pi: \mathcal{S} \to \mathcal{S}$ and single-step reward function $r_\pi: \mathcal{S} \to \R$ are defined as:
\begin{align}
    f_\pi(s) &= f(s, \pi(s)), & r_\pi(s) = r(s, \pi(s)).
\end{align}
for all states $s \in \mathcal{S}$. For any $t \in \Z_+$, let $f_\pi^t: \mathcal{S} \to \mathcal{S}$ denote the $t$-iterated composition of $f_\pi$. Define the subset $\mathcal{S}_\pi \subseteq \mathcal{S}$ as:
\begin{equation}
    \mathcal{S}_\pi = \left\{ s \in \mathcal{S}: \sum_{t = 0}^\infty \gamma^t r_\pi( f_\pi^t(s) ) ~\mathrm{converges~absolutely} \right\}.
\end{equation}
If $r_\pi$ is bounded, then $\mathcal{S}_\pi = \mathcal{S}$. If $\gamma > 0$, then $f_\pi(\mathcal{S}_\pi) \subseteq \mathcal{S}_\pi$, as a convergent series still converges after the first term is removed. If $\mathcal{S}_\pi \neq \emptyset$, then the corresponding value function $V_\pi: \mathcal{S}_\pi \to \R$ is defined explicitly and implicitly, for all states $s \in \mathcal{S}_\pi$, as:
\begin{equation}
    V_\pi(s) = \sum_{t = 0}^\infty \gamma^t r_\pi(f_\pi^t(s)) = r_\pi(s) + \gamma V_\pi(f_\pi(s)).
\end{equation}

A set of policies $\Pi$ admits a partial order if $\mathcal{S}_\pi = \mathcal{S}$ for all policies $\pi \in \Pi$. In this case, for policies $\pi, \pi' \in \Pi$, if $V_{\pi}(s) \geq V_{\pi'}(s)$ for all states $s \in \mathcal{S}$, then $\pi \succeq \pi'$.

\begin{definition}[Metric Policy Optimization Problem]
We refer to the pair of an MDP $(\mathcal{S}, \mathcal{A}, C, f, r, \gamma)$ and a policy class $\Pi$ admitting a partial order as a \textit{metric policy optimization problem} (MPOP). We call an MPOP well-posed if there is an optimal policy $\pi^* \in \Pi$ with $\pi^* \succeq \pi$ for all policies $\pi \in \Pi$.
\end{definition}

\textbf{Goals of this section.} Our goal is to analyze well-posedness of policy optimization for MDPs. We first describe sufficient conditions for well-posedness of value iteration (Theorem \ref{thm:well-posed}). We then establish conditions under which ill-posed MPOPs can be restricted to well-posed problems by only considering policies that render the same subset of the state space forward-invariant (Theorem \ref{thm:invariance}), resulting in CR-MPOPs. We will show how to apply our results on value iteration to control affine systems in Section \ref{sec:control}.

Throughout this section, we refer to and modify the following example to ground our presentation:
\begin{example}\label{ex:running-example}
Consider the state space $\mathcal{S} = \R$, action space $\mathcal{A} = \R$, and the constant action-admissibility correspondence $C = \mathcal{A}$. Additionally, consider the transition function $f: \mathrm{Graph}(C) \to \mathcal{S}$ and reward function $r: \mathrm{Graph}(C) \to \R$ defined as:
\begin{align}
    f(s, a) &= s + \tanh{(a)}, & r(s, a) = -s^2 - (\tanh{(a)})^2,
\end{align}
for all state-action pairs $(s, a) \in \mathrm{Graph}(C)$, and a discount factor $\gamma \in [0, 1)$. Note that $| f(s, a) | \leq | s | + | \tanh{(a)} | \leq | s | + 1$ for all state-action pairs $(s, a) \in \mathrm{Graph}(C)$. Consider any policy $\pi: \mathcal{S} \to \mathcal{A}$. By induction, we have $| f_\pi^t(s) | \leq | s | + t$ for all states $s \in \mathcal{S}$ and $t \in \Z_+$. For any state $s \in \mathcal{S}$ and $T \in \Z_+$, we have: 
\begin{align}
    \sum_{t = 0}^T \gamma^t r_\pi(f_\pi^t(s)) &\geq -\sum_{t = 0}^\infty \gamma^t ( (f_\pi^t(s))^2 + (\tanh{(\pi(f_\pi^t(s)))})^2 )\nonumber\\
    &\geq -\sum_{t = 0}^\infty \gamma^t ( s^2 + 2|s|t + t^2 + 1 )\nonumber\\
    &= -\frac{s^2 + 1}{1 - \gamma} - \frac{2\gamma |s|}{(1 - \gamma)^2} - \frac{\gamma(\gamma + 1)}{(1 - \gamma)^3}.
\end{align}
The sequence of partial sums as $T \to \infty$ is monotone and bounded below, so $V_\pi(s)$ is well-defined. As $s$ was arbitrary, we have $\mathcal{S}_\pi = \mathcal{S}$.
\end{example}

\subsection{Value Iteration}
Consider a policy class $\Pi$ with $\mathcal{S}_\pi = \mathcal{S}$ for each policy $\pi \in \Pi$ and $\sup_{\pi \in \Pi} V_\pi(s) < \infty$ for each state $s \in \mathcal{S}$. Accordingly, we define the optimal (with respect to $\Pi$) value function $V^*: \mathcal{S} \to \R$ as:
\begin{equation}
    V^*(s) = \sup_{\pi \in \Pi} V_\pi(s),
\end{equation}
for all states $s \in \mathcal{S}$. If a policy $\pi^* \in \Pi$ satisfies $V_{\pi^*} = V^*$, then $\pi^*$ is optimal and the policy optimization problem is well-posed.

Value iteration generates a sequence of approximations of $V^*$. Given an initial guess $V_0: \mathcal{S} \to \R$, we seek a sequence of real-valued functions $\{ V_n : n \in \N \}$ satisfying:
\begin{equation}\label{eq:value-iteration-step}
    V_{n + 1}(s) = \sup_{a \in C(s)}{\{ r(s, a) + \gamma V_n(f(s, a)) \}},
\end{equation}
for all states $s \in \mathcal{S}$ and $n \in \Z_+$.

We now describe conditions under which MPOPs are well-posed and can be solved with value iteration (Theorem \ref{thm:well-posed}). When assumptions are strengthened by requiring $\mathcal{S}$ and $\mathcal{A}$ to be Polish spaces, we will make use of \cite[Theorem 4.1]{feinberg2012average} (the setting of this theorem is called a semicontinuous model in \cite[\S 6.7]{dynkin1979controlled}). We provide a more direct proof of well-posedness without these additional assumptions in the appendix.

\begin{assumption}\label{as:well-posed}
The action admissibility correspondence $C$ has compact values and is upper hemicontinuous, the transition function $f$ is continuous, the reward function $r$ is upper semicontinuous and bounded, and the policy class $\Pi$ is the set of all measurable selectors of $C$. Moreover, for each state-action pair $(s, a) \in \mathrm{Graph}(C)$, there is a corresponding policy $\pi_{s, a} \in \Pi$ satisfying $\pi_{s, a}(s) = a$.
\end{assumption}

\begin{remark}\label{rem:relax-assumption}
If the projection of $\mathrm{Graph}(C)$ onto the action space is compact, then the assumption that for any state-action pair $(s, a) \in \mathrm{Graph}(C)$, there exists a policy $\pi_{s, a} \in \Pi$ with $\pi_{s, a}(s) = a$ can be removed. See Lemma \ref{lem:relax-assumption}.
\end{remark}

\begin{remark}[Tabular MDP]
If $\mathcal{S}$ and $\mathcal{A}$ are finite sets equipped with discrete metrics and $\Pi$ is the set of all selectors of $C$, then these assumptions are immediately met. This follows since the discrete metric renders all sets are open, closed, and compact, and all functions defined on such sets are continuous (and bounded if real-valued).
\end{remark}

\begin{example}\label{ex:running-example-well-posed}
Consider the MDP $(\mathcal{S}, \mathcal{A}, C, f, r, \gamma)$ from Example \ref{ex:running-example}. While $C$ is upper hemicontinuous, it is not compact-valued. While $f$ and $r$ are continuous, $r$ is unbounded. Consider the constant correspondence $C' = [-1, 1]$, and the reward $r': \mathrm{Graph}(C') \to \R$:
\begin{equation}
    r'(s) = e^{-s^2} - (\tanh{(a)})^2,
\end{equation}
for all state-action pairs $(s, a) \in \mathrm{Graph}(C')$. Note that $\mathrm{Graph}(C') = \mathcal{S} \times [-1, 1]$, so $r'$ is bounded. Since $f$ is continuous, its restriction to $\mathrm{Graph}(C')$ is as well. With $\Pi$ the set of all measurable selectors of $C'$, note that for every pair $(s, a) \in \mathrm{Graph}(C')$, the constant policy $\pi_{s, a} = a$ is a measurable selector of $C'$, so $\pi_{s, a} \in \Pi$. Thus, the MPOP defined by the MDP $(\mathcal{S}, \mathcal{A}, C', \left.f\right|_{\mathrm{Graph}(C')}, r', \gamma)$ and policy class $\Pi$ satisfies Assumption \ref{as:well-posed}.
\end{example}

We now demonstrate well-posedness under Assumption \ref{as:well-posed}.
\begin{theorem}\label{thm:well-posed}
Consider an MPOP characterized by an MDP $(\mathcal{S}, \mathcal{A}, C, f, r, \gamma)$ and a policy class $\Pi$ that satisfies Assumption \ref{as:well-posed}. There is an optimal policy $\pi^* \in \Pi$ satisfying $V_{\pi^*} = V^*$, and $V^*$ is the limit of value iteration when the initial guess is bounded and upper semicontinuous.
\end{theorem}

\begin{proof}
If $\mathcal{S}$ and $\mathcal{A}$ are Polish spaces, we can show that the MDP in question satisfies the three conditions of Assumptions (W*) of \cite{feinberg2012average}, in which case, there is an optimal policy and $V^*$ is the limit of value iteration when the initial condition is the zero function (a conclusion of \cite[Theorem 4.1]{feinberg2012average}). The first condition follows since $r$ is upper semicontinuous and bounded above (as it is bounded both above and below), which corresponds to a lower semicontinuous cost function that is bounded below. The second condition is satisfied by the upper hemicontinuity of $C$ (using the sequential definition in \cite[Theorem 17.20]{aliprantis2006infinite}) since $r$ is also bounded below. The third condition is satisfied since the MDP is deterministic (for which transition probabilities are represented by Dirac measures). In this case, we require:
\begin{equation}
    \lim_{n \to \infty}{ \varphi(f(s_n, a_n)) } = \varphi(f(s, a)),
\end{equation}
for any bounded continuous function $\varphi: \mathcal{S} \to \R$ and any convergent sequence $\{ (s_n, a_n) \in \mathrm{Graph}(C): n \in \N \}$ converging to $(s, a) \in \mathrm{Graph}(C)$; equality follows since $\varphi \circ f$ is continuous (as the composition of continuous functions).

If $\mathcal{S}$ and $\mathcal{A}$ are not both Polish spaces, we prove well-posedness in the appendix, in which case, $V^*$ is the limit of value iteration when the initial guess is any bounded and upper semicontinuous function.
\end{proof}

We conclude with a lemma mentioned in Remark \ref{rem:relax-assumption}, allowing Assumption \ref{as:well-posed} to be relaxed slightly.

\begin{lemma}\label{lem:relax-assumption}
Suppose an MDP $(\mathcal{S}, \mathcal{A}, C, f, r, \gamma)$ and policy class $\Pi$ satisfy Assumption \ref{as:well-posed}, suppose the projection of $\mathrm{Graph}(C)$ onto $\mathcal{A}$ is compact, and let $\rho_{\mathcal{A}}: \mathcal{A} \times \mathcal{A} \to \R_+$ denote the metric on $\mathcal{A}$. For any state-action pair $(s, a) \in \mathrm{Graph}(C)$, the function $\varphi_{s, a}: \mathrm{Graph}(C) \to \R$ defined as:
\begin{equation}
    \varphi_{s, a}(s', a') = -\rho_{\mathcal{A}}(a', a),
\end{equation}
for all state-action pairs $(s', a') \in \mathrm{Graph}(C)$ admits a maximizing policy $\pi_{s, a} \in \Pi$ for which $\pi_{s, a}(s) = a$.
\end{lemma}

\begin{proof}
Since $\rho_{\mathcal{A}}$ is continuous, so is $\varphi_{s, a}$. Since the projection of $\mathrm{Graph}(C)$ onto $\mathcal{A}$ is compact, $\varphi_{s, a}$ is bounded. If $\mathcal{S}$ and $\mathcal{A}$ are Polish spaces, then the proof of \cite[Lemma 3.4]{feinberg2012average} uses the Arsenin-Kunugui theorem to show the existence of a policy $\pi_{s, a} \in \Pi$ satisfying:
\begin{align}\label{eq:distance-minimizer}
    \varphi_{s, a}(s', \pi_{s, a}(s')) &= \max_{a' \in C(s')} \varphi_{s, a}(s', a')\nonumber\\
    &= -\min_{a' \in C(s')} \rho_{\mathcal{A}}(a', a),
\end{align}
for all states $s' \in \mathcal{S}$. In particular, $\pi_{s, a}(s) = a$ as $\varphi_{s, a}(s, \pi_{s, a}(s)) = -\min_{a' \in C(s)} \rho_{\mathcal{A}}(a', a) = 0$. This requires Assumptions (W*) of \cite{feinberg2012average} to be met, which is verified in the proof of Theorem \ref{thm:well-posed}.

We arrive at the same conclusion when $\mathcal{S}$ and $\mathcal{A}$ are not both Polish spaces in the appendix.
\end{proof}

\subsection{Compactly Restrictable MPOPs}
\label{sec:cr-mmdps}
While Assumption \ref{as:well-posed} allows us to determine when MPOPs are well-posed and can be solved using value iteration, it is not met for many systems of interest. We now outline an alternative assumption that not only mitigates theoretical shortcomings, but also better fits the problem settings for physical systems. Our main result here is Theorem \ref{thm:invariance}, showing how an MPOP satisfying these new assumptions can generate an MPOP satisfying Assumption \ref{as:well-posed} through a systematic transformation. We call the original MPOP a CR-MPOP.

\subsubsection{Motivation}
Consider an MPOP defined by an MDP $(\mathcal{S}, \mathcal{A}, C, f, r, \gamma)$ and a policy class $\Pi$ that satisfies Assumption \ref{as:well-posed}. The proof of Theorem \ref{thm:well-posed} requires a bounded reward function $r$ and a compact-valued action-admissibility correspondence $C$. For many examples, these requirements are too strict, as shown in Example \ref{ex:running-example-well-posed}. If $C$ can be modified such that its graph is compact and if $r$ is continuous, then $r$ will be bounded over the graph of the modified correspondence. This guides our systematic construction of a well-posed policy optimization problem from an ill-posed one.

If the graph of the modified correspondence is compact, the projection of the graph onto the state space must also be compact. Therefore, if the state space is not already compact, then the modified correspondence must be defined on a compact strict subset $\mathcal{S}_0 \subset \mathcal{S}$; that is, the modified correspondence must have the form $C_0: \mathcal{S}_0 \to \mathcal{P}(\mathcal{A})$, where $C_0(s) \subseteq C(s)$ is nonempty and compact for all states $s \in \mathcal{S}_0$. Moreover, while restricting admissible actions to compact sets is reasonable for physical systems (as forces, voltages, currents, etc. are bounded by physical constraints), these restrictions must be chosen carefully to enforce the nonemptiness condition. For the restriction of $f$ to $\mathrm{Graph}(C_0)$ to be a well-defined transition function, we require $f(s, a) \in \mathcal{S}_0$ for all pairs $(s, a) \in \mathrm{Graph}(C_0)$. Finally, we must ensure that $C_0$ admits measurable selectors.

Restricting our attention to a compact subset $\mathcal{S}_0$ is also reasonable for many physical systems, as dynamics models are often well-understood only in a bounded set around an operating condition.

\subsubsection{Compact Restriction}
For this construction, we use policies that render compact subsets of the state space forward-invariant. A policy $\pi \in \Pi$ renders a subset $\mathcal{S}_0 \subseteq \mathcal{S}$ forward-invariant if $f_\pi(\mathcal{S}_0) \subseteq \mathcal{S}_0$. By induction, we have $f_\pi^t(\mathcal{S}_0) \subseteq \mathcal{S}_0$ for any $t \in \Z_+$.

\begin{definition}[Compactly Restrictable Metric Policy Optimization Problem]
An MPOP characterized by an MDP $(\mathcal{S}, \mathcal{A}, C, f, r, \gamma)$ and a policy class $\Pi$ is a \textit{compactly restrictable metric policy optimization problem} (CR-MPOP) if there is a nonempty compact subset $\mathcal{S}_0 \subseteq \mathcal{S}$ and a continuous policy $\pi_0 \in \Pi$ rendering $\mathcal{S}_0$ forward-invariant.
\end{definition}

\begin{assumption}\label{as:general}
The action-admissibility correspondence $C$ has closed values and $\mathrm{Graph}(C)$ is closed, the transition function $f$ and reward function $r$ are continuous, and the policy class $\Pi$ is the set of all measurable selectors of $C$.
\end{assumption}

To show how an CR-MPOP satisfying Assumption \ref{as:general} can lead to an MPOP satisfying Assumption \ref{as:well-posed}, we need the following lemma.

\begin{lemma}\label{lem:continuous-on-graph-implies-continuous-sections}
Consider an MDP $(\mathcal{S}, \mathcal{A}, C, f, r, \gamma)$ and a policy class $\Pi$ satisfying Assumption \ref{as:general}. For any state $s \in \mathcal{S}$, the $s$-section $f_s: C(s) \to \mathcal{S}$ is continuous. Additionally, for any state $s \in \mathcal{S}$ and any closed set $G \subseteq \mathcal{S}$, the preimage $(f_s)^{-1}(G)$ is a closed subset of $\mathcal{A}$.
\end{lemma}

\begin{proof}
Let $\mathrm{id}: \mathcal{S} \times \mathcal{A} \to \mathcal{S} \times \mathcal{A}$ denote the identity function on $\mathcal{S} \times \mathcal{A}$. For any state $s \in \mathcal{S}$, the $s$-section $\mathrm{id}_s: \mathcal{A} \to \mathcal{S} \times \mathcal{A}$ is continuous, meaning the restriction $\left.\mathrm{id}_s\right|_{C(s)}: C(s) \to \mathcal{S} \times \mathcal{A}$ is continuous. We can write $f_s = f \circ (\left.\mathrm{id}_s\right|_{C(s)})$, implying $f_s$ is continuous. For any closed set $G \subseteq \mathcal{S}$, there is a corresponding closed set $F_{s, G} \subseteq \mathcal{A}$ satisfying $(f_s)^{-1}(G) = C(s) \cap F_{s, G}$. Since $C(s)$ is also a closed subset of $\mathcal{A}$, we conclude that $(f_s)^{-1}(G)$ a closed subset of $\mathcal{A}$.
\end{proof}

We now present our guarantee of the existence of a well-posed value iteration settings under Assumption \ref{as:general}.

\begin{theorem}
\label{thm:invariance}
Consider a CR-MPOP characterized by an MDP $(\mathcal{S}, \mathcal{A}, C, f, r, \gamma)$ and a policy class $\Pi$ that satisfies Assumption \ref{as:general}. There exists an MPOP characterized by an MDP $(\mathcal{S}_0, \mathcal{A}_0, C_0, f_0, r_0, \gamma)$ and a policy class $\Pi_0$ satisfying Assumption \ref{as:well-posed}, with $\mathcal{S}_0 \subseteq \mathcal{S}$ rendered forward-invariant by a continuous policy in $\Pi$, $\mathcal{A}_0 \subseteq \mathcal{A}$, $C_0: \mathcal{S}_0 \to \mathcal{P}(\mathcal{A}_0)$ satisfying $C_0(s) \subseteq C(s)$ for all states $s \in \mathcal{S}_0$, $f_0 = \left.f\right|_{\mathrm{Graph}(C_0)}$, $r_0 = \left.r\right|_{\mathrm{Graph}(C_0)}$, and:
\begin{equation}
    \Pi_0 = \{ \left.\pi\right|_{\mathcal{S}_0} : \pi \in \Pi, \pi(\mathcal{S}_0) \subseteq \mathcal{A}_0, ~\mathrm{and}~ f_\pi(\mathcal{S}_0) \subseteq \mathcal{S}_0 \}.
\end{equation}
\end{theorem}

\begin{proof}
For some $n \in \N$, consider continuous policies $\pi_1, \dots, \pi_n \in \Pi$ satisfying $f_{\pi_i}(\mathcal{S}_0) \subseteq \mathcal{S}_0$ for all $i \in \{ 1, \dots, n \}$. Such policies are guaranteed to exist by the existence of the continuous policy $\pi_0 \in \Pi$. Consider the union of the images of $\mathcal{S}_0$ under each of these policies, defining the compact set $\mathcal{A}_0 = \bigcup_{i = 1}^n \pi_i(\mathcal{S}_0)$.
Since $\mathcal{S}$ and $\mathcal{A}$ are separable metric spaces, so are $\mathcal{S}_0$ and $\mathcal{A}_0$ \cite[Corollary 3.5]{aliprantis2006infinite}.

Define the correspondence $C_0: \mathcal{S}_0 \to \mathcal{P}(\mathcal{A}_0)$ as:
\begin{equation}
    C_0(s) = \{ a \in C(s) \cap \mathcal{A}_0: f(s, a) \in \mathcal{S}_0 \} = (f_s)^{-1}(\mathcal{S}_0) \cap \mathcal{A}_0,
\end{equation}
for all states $s \in \mathcal{S}_0$. 
Since $\pi_1(s), \dots, \pi_n(s) \in C_0(s)$, $C_0(s) \neq \emptyset$  for all $s \in \mathcal{S}_0$.
By Lemma \ref{lem:continuous-on-graph-implies-continuous-sections} and since $\mathcal{S}_0$ is closed (it is compact in a metric space), the preimage $(f_s)^{-1}(\mathcal{S}_0)$ is a closed subset of $\mathcal{A}$. $C_0(s)$ is thus compact, as it is a closed subset of $\mathcal{A}_0$. For any closed set $G \subseteq \mathcal{A}_0$, the lower preimage of $G$ under $C_0$ satisfies:
\begin{align}
    &~(C_0)^\ell(G) = \{ s \in \mathcal{S}_0: f(s, a) \in \mathcal{S}_0 ~\mathrm{for~some}~ a \in C(s) \cap G \}\nonumber\\
    &~\qquad= p_\mathcal{S}( f^{-1}(\mathcal{S}_0) \cap (\mathcal{S}_0 \times G) ),
\end{align}
where $p_{\mathcal{S}}: \mathcal{S} \times \mathcal{A} \to \mathcal{S}$ denotes the canonical projection onto $\mathcal{S}$. Since $\mathcal{S}_0$ is closed and $f$ is continuous, there is a closed set $F \subseteq \mathcal{S} \times \mathcal{A}$ satisfying $f^{-1}(\mathcal{S}_0) = \mathrm{Graph}(C) \cap F$. Since $\mathrm{Graph}(C)$ is also closed, $f^{-1}(\mathcal{S}_0)$ is a closed subset of $\mathcal{S} \times \mathcal{A}$. Since $G$ is a closed subset of the compact set $\mathcal{A}_0$, $G$ is compact, and since $\mathcal{S}_0$ is compact, the product $\mathcal{S}_0 \times G$ is compact. As a closed subset of a compact set, $f^{-1}(\mathcal{S}_0) \cap (\mathcal{S}_0 \times G)$ is compact. Since the projection operator $p_{\mathcal{S}}$ is continuous, $(C_0)^\ell(G)$ is compact. Therefore, $(C_0)^\ell(G)$ is a closed subset of $\mathcal{S}_0$, and since $G$ was arbitrary, $C_0$ is upper hemicontinuous.

As restrictions of continuous functions, $f_0$ and $r_0$ are continuous, implying $r_0$ is upper semicontinuous. Note that:
\begin{align}
    \mathrm{Graph}(C_0) &= \{ (s, a) \in \mathcal{S}_0 \times \mathcal{A}_0: f(s, a) \in \mathcal{S}_0 \}\nonumber\\
    &= f^{-1}(\mathcal{S}_0) \cap (\mathcal{S}_0 \times \mathcal{A}_0).
\end{align}
Since $\mathcal{A}_0$ is trivially a closed subset of $\mathcal{A}_0$, we have already shown that $f^{-1}(\mathcal{S}_0) \cap (\mathcal{S}_0 \times \mathcal{A}_0)$ is compact, meaning $\mathrm{Graph}(C_0)$ is compact. Therefore, since $r_0$ is continuous, it is bounded. 

The policy class $\Pi_0$ satisfies:
\begin{align}
    &\Pi_0 = \{ \left.\pi\right\vert_{\mathcal{S}_0}: \pi \in \Pi, \pi(s) \in C_0(s) ~\mathrm{for~all~states}~ s \in \mathcal{S}_0 \}\nonumber\\
    &= \{ \pi \in \mathcal{L}^0(\mathcal{S}_0; \mathcal{A}_0): \pi(s) \in C_0(s) ~\mathrm{for~all~states}~ s \in \mathcal{S}_0 \}.
\end{align}
To verify the last equality, first consider a policy $\pi \in \Pi$ satisfying $\pi(s) \in C_0(s)$ for all states $s \in \mathcal{S}_0$. The restriction $\left.\pi\right|_{\mathcal{S}_0}$ is a measurable function from $\mathcal{S}_0$ to $\mathcal{A}_0$ selecting from $C_0$. Conversely, consider a measurable function $\pi: \mathcal{S}_0 \to \mathcal{A}_0$ selecting from $C_0$. Pick any policy $\pi_e \in \Pi$ and define the extension $\bar{\pi}: \mathcal{S} \to \mathcal{A}$
as $\pi$ on $\mathcal{S}_0$ and $\pi_e$ on $\mathcal{S} \setminus \mathcal{S}_0$. For any measurable set $B \subseteq \mathcal{A}$, we have:
\begin{align}
    &\bar{\pi}^{-1}(B) = \{ s \in \mathcal{S}_0: \pi(s) \in \mathcal{A}_0 \cap B \} \cup \{ s \in \mathcal{S} \setminus \mathcal{S}_0: \pi_e(s) \in B \}\nonumber\\
    &~= \pi^{-1}(\mathcal{A}_0 \cap B) \cup (\pi_e^{-1}(B) \setminus \mathcal{S}_0),
\end{align}
which is a measurable set; therefore, $\bar{\pi}$ is a measurable function. Note that $\bar{\pi}(s) = \pi(s) \in C_0(s) \subseteq C(s)$ for all states $s \in \mathcal{S}_0$ and $\bar{\pi}(s) = \pi_e(s) \in C(s)$ for all states $s \notin \mathcal{S}_0$. Since $\bar{\pi}$ is a measurable selector of $C$, we have $\bar{\pi} \in \Pi$, and inclusion follows since $\left.\bar{\pi}\right|_{\mathcal{S}_0} = \pi$.

Finally, since $\mathrm{Graph}(C_0)$ is compact, its projection onto $\mathcal{A}_0$ is compact. Therefore, by Lemma \ref{lem:relax-assumption}, for every state-action pair $(s, a) \in \mathrm{Graph}(C_0)$, there is a policy $\pi_{s, a} \in \Pi_0$ satisfying $\pi_{s, a}(s) = a$.
\end{proof}

\begin{remark}
An MPOP characterized by an MDP $(\mathcal{S}, \mathcal{A}, C, f, r, \gamma)$ and a policy class $\Pi$ that satisfies Assumption \ref{as:well-posed} is trivially a CR-MPOP satisfying Assumption \ref{as:general} if the following sufficient conditions are met:
\begin{inparaenum}[1)]
    \item The state space $\mathcal{S}$ is compact,
    \item The reward function $r$ is continuous,
    \item The policy class $\Pi$ contains a continuous policy.
\end{inparaenum}
\end{remark}

\begin{example}
Consider the MDP $(\mathcal{S}, \mathcal{A}, C, f, r, \gamma)$ defined in Example \ref{ex:running-example} with a policy class $\Pi$ comprised of measurable selectors of $C$, and let $\mathcal{S}_0 = [-1, 1]$. Note that $| f(0, a) | = | \tanh{(a)} | < 1$ for all actions $a \in \R$. Consider any state $s \in [-1, 0)$. For any nonnegative action $a \in \R_+$, we have $-1 \leq s \leq f(s, a) = s + \tanh{(a)} < s + 1 < 1$, and $f(s, -a) = s + \tanh{(-a)} \in [-1, s)$ if and only if $-a \in [\tanh^{-1}{(-1 - s)}, 0)$ since $\tanh^{-1}$ is monotonically increasing. By a similar argument for states in $(0, 1]$, we determine that, for all $s \in \mathcal{S}_0$ and $a \in \mathcal{A}$, we have $f(s, a) \in \mathcal{S}_0$ if and only if:
\begin{equation}\label{eq:fi-condition}
    a \in \begin{cases} [\tanh^{-1}{(-1 - s)}, \infty) & -1 \leq s < 0,\\
    \mathcal{A} & s = 0,\\
    (-\infty, \tanh^{-1}{(1 - s)}] & 0 < s \leq 1.
    \end{cases}
\end{equation}

\begin{figure}
    \centering
    \includegraphics[width=0.35\textwidth]{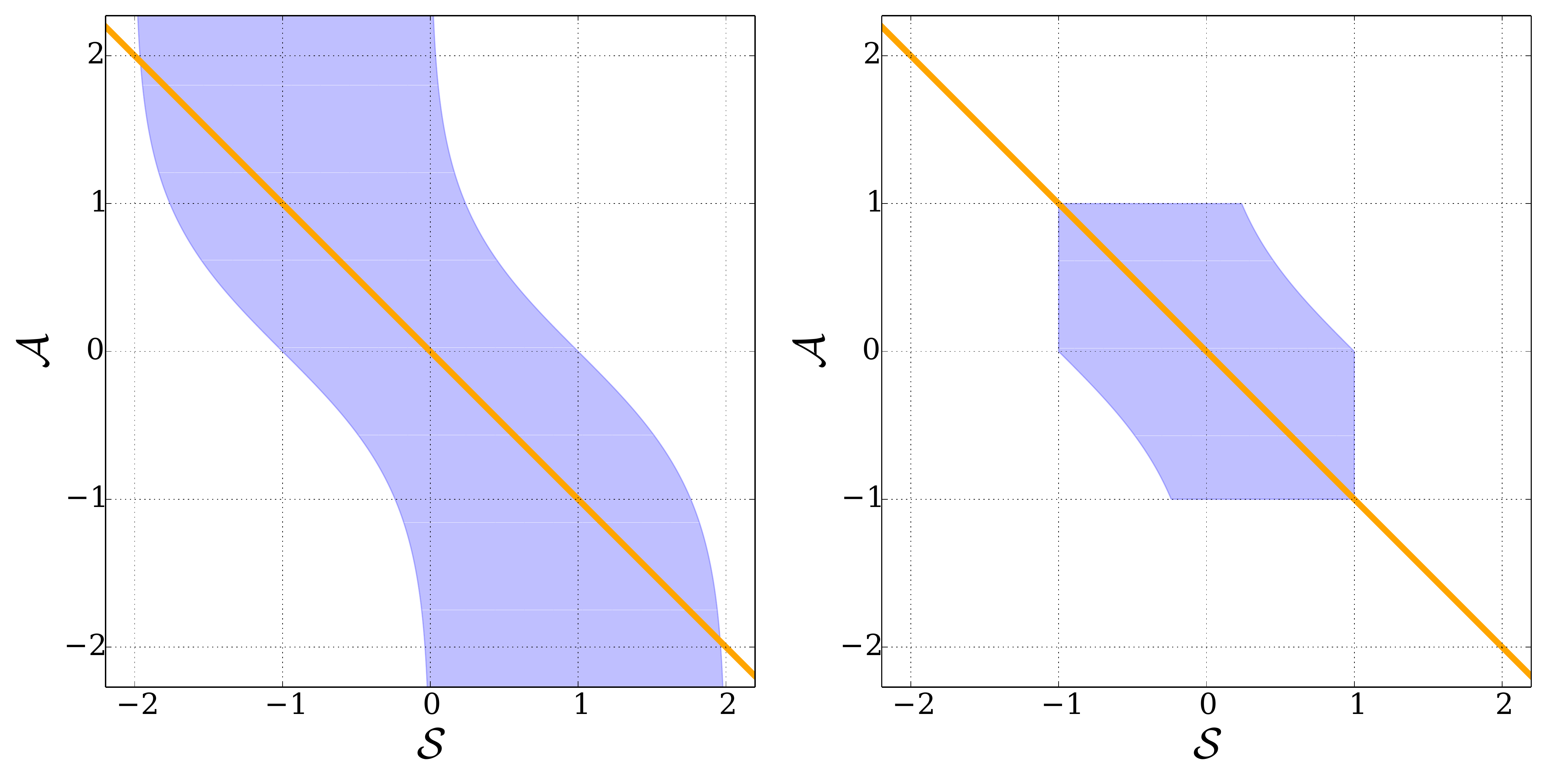}
    \caption{(Left) The preimage $f^{-1}(\mathcal{S}_0)$ is shown in blue, which contains all state-action pairs mapped into $\mathcal{S}_0 = [-1, 1]$ by $f$. For any state in $\mathcal{S}_0$, the set of actions mapping that state into $\mathcal{S}_0$ is unbounded, therefore not compact. (Right) The graph of $\pi_0$ is shown in orange and the graph of $C_0$ is shown in blue. The graph of $C_0$ is compact, and for every state $s \in \mathcal{S}_0$, the set $C_0(s)$ is compact.}
    \label{fig:compact-construction}
\end{figure}

A correspondence coinciding with the requirements on actions in \eqref{eq:fi-condition} would not have compact values; this is illustrated by the preimage $f^{-1}(\mathcal{S}_0)$ in Figure \ref{fig:compact-construction}. Therefore, consider the continuous policy $\pi_0 \in \Pi$ defined as $\pi_0(s) = -s$ for all states $s \in \mathcal{S}$. For any state $s \in [-1, 0)$, we have $\pi_0(s) = -s > 0 \geq \tanh^{-1}(-1 - s)$.
Similarly, for any state $s \in (0, 1]$, we have $\pi_0(s) < \tanh^{-1}(1 - s)$. Therefore, $f_{\pi_0}(\mathcal{S}_0) \subseteq \mathcal{S}_0$. The image of $\mathcal{S}_0$ under $\pi_0$ is $[-1, 1]$. Therefore, we can define $\mathcal{A}_0 = [-1, 1]$ and $C_0: \mathcal{S}_0 \to \mathcal{P}(\mathcal{A}_0)$ as, for all states $s \in \mathcal{S}_0$:
\begin{equation}
    C_0(s) = \begin{cases}
        [\tanh^{-1}(-1 - s), 1] & s \leq -1 - \tanh{(-1)},\\
        [-1, \tanh^{-1}(1 - s)] & s \geq 1 - \tanh{(1)},\\
        [-1, 1] & \mathrm{otherwise}.
    \end{cases}
\end{equation}
\end{example}

\subsection{Policy Iteration}

We briefly mention the difficulties of policy iteration for MPOPs, as noted in \cite{yu2015mixed}. For a well-posed MPOP, policy iteration generates a monotonically nondecreasing sequence of policies. Given an initial policy $\pi_0 \in \Pi$, we seek a sequence of policies $\{ \pi_n \in \Pi: n \in \N \}$ satisfying, for all states $s \in \mathcal{S}$ and $n \in \Z_+$:
\begin{align}\label{eq:policy-improvement-step}
    &r(s, \pi_{n + 1}(s)) + \gamma V_{\pi_n}(f(s, \pi_{n + 1}(s)))\nonumber\\
    &~\qquad\qquad= \sup_{a \in C(s)}{\{ r(s, a) + \gamma V_{\pi_n}(f(s, a)) \}}.
\end{align}
If $f$, $r$, and $\pi_0$ are continuous, then so are $f_{\pi_0}$ and $r_{\pi_0}$ which we can show renders $V_{\pi_0}$ continuous. This implies $r + \gamma V_{\pi_0} \circ f$ is continuous. If $r$ is bounded, then so are $V_{\pi_0}$ and $r + \gamma V_{\pi_0} \circ f$. Since $r + \gamma V_{\pi_0} \circ f$ is upper semicontinuous and bounded, we can show (using \cite{feinberg2012average} or the appendix) that for any state $s \in \mathcal{S}$, the optimization problem:
\begin{equation}
    \sup_{a \in C(s)}{\{ r(s, a) + \gamma V_{\pi_0}(f(s, a)) \}},
\end{equation}
is solved by the next policy $\pi_1$, a measurable selector of $C$. However, we cannot conclude that $r + \gamma V_{\pi_1} \circ f$ is upper semicontinuous and bounded, so the same argument cannot be applied iteratively.

Demonstrating that policy iteration can be applied requires special knowledge that the policy improvement step in \eqref{eq:policy-improvement-step} continues to produce policies with value functions that permit maximization and appropriate selection (for the chosen policy class). Such knowledge is available in linear-quadratic regulator (LQR) problems with linear policy classes (see \cite{tu2019sample} for the discounted case).

\subsection{Computation}
There are 3 central difficulties in computationally implementing value iteration for a well-posed MPOP:
\begin{inparaenum}[1)]
    \item Representing iterates during value iteration,
    \item Determining admissible actions, and
    \item Approximating the update step \eqref{eq:value-iteration-step}.
\end{inparaenum}
Typically, iterates are represented using function approximators such as neural networks \cite{sutton2018reinforcement} or Gaussian process regression models \cite{kuss2003gaussian}. Training such approximators involves sampling sufficiently many states from the state space region of interest. For any given state, the maximization in the update step \eqref{eq:value-iteration-step} is generally a nonconvex optimization, with possible approximations including (projected) gradient ascent from several initial action seeds or maximization over a finite but sufficiently dense sampling of admissible actions. If the function approximation and optimization approximation errors can be controlled, approximate convergence of value iteration can be guaranteed \cite[Proposition 2.3.2]{bertsekas2022abstract}.

Any approximate optimization approach requires checking admissibility of actions, which, in the case of a CR-MPOP, requires checking whether or not a state-action pair is mapped to the correct compact set under the transition map. This necessitates efficient set membership checking, and for nonconvex compact sets, proximity-based membership approximations may be needed.

Finally, in Section \ref{sec:sampled-data-control}, we will use the closure of a reachable set (under a specific policy) to restrict problems with feedback linearizable control affine systems to well-posed MPOPs. To sample a state from the reachable set, we can sample a state from the appropriate set of initial conditions and follow the policy for a number of steps sampled from a geometric distribution (with success probability $1 - \gamma$). The resulting distribution is the discounted state distribution under the policy \cite{silver2014deterministic}, which is supported on the entirety of the reachable set (thus, on a dense subset of its closure).

There are many possible combinations of computational approaches that may be suitable for well-posed MPOPs; we deem the precise best choice outside the scope of this theoretical framing.

\section{Control Affine Systems}
\label{sec:control}
We now show how to apply the results from Section \ref{sec:mdp} to a general class of control systems called control affine systems \cite[\S 13.2.3]{lavalle2006planning}. For state and action space dimensions $d, m \in \N$, respectively, consider a set $\mathcal{D} \subseteq \R^d$ and vector fields $f_0, g_1, \dots, g_m: \mathcal{D} \to \R^d$. Define the matrix-valued function $G: \mathcal{D} \to \R^{d \times m}$ with columns $g_1, \dots, g_m$. Define $F: \mathcal{D} \times \R^m \to \R^d$ as:
\begin{equation}
    F(x, u) = f_0(x) + G(x)u,
\end{equation}
for all $x \in \mathcal{D}$ and $u = (u_1, \dots, u_m) \in \R^m$. An initial value problem with constant control input $u \in \R^m$ is characterized by an initial condition $x \in \mathcal{D}$ and an open time interval $I \subseteq \R$ with $0 \in I$; a corresponding solution is a differentiable function $\phi: I \to \mathcal{D}$ satisfying $\phi(0) = x$ and $\dot{\phi}(t) = \frac{\mathrm{d}}{\mathrm{d}t}\phi(t) = F(\phi(t), u) = f_0(\phi(t)) + G(\phi(t))u$
for all times $t \in I$.

\subsection{Time Sampling}
In contrast to continuous-time control design, we consider sampled-data control design (see \cite{nevsic1999sufficient, taylor2021sampled, taylor2022safety}), in which initial value problems with constant control characterize the evolution of a system over fixed sample intervals, resulting in control input trajectories that are piecewise constant in time. Such assumptions are realistic for physical systems interacting with digital controllers, which measure states and compute control inputs at nearly fixed frequencies. This setting requires the time intervals over which solutions are defined to be sufficiently long, uniformly for all initial conditions and control inputs under consideration. 

Specifically, fix a sample period $h \in \R_{++}$, and define the subset $\mathcal{S}_h \subseteq \mathcal{D}$ and correspondence $C_h: \mathcal{S}_h \to \mathcal{P}(\R^m)$ such that the following properties are satisfied:
\begin{itemize}
    \item For every initial condition $x \in \mathcal{S}_h$, the corresponding set of control inputs $C_h(x) \subseteq \R^m$ is nonempty,
    \item For every  $(x, u) \in \mathrm{Graph}(C_h)$, there is a unique solution to the initial value problem characterized by initial condition $x$, control input $u$, and open time interval $I \subseteq \R$  with $[0, h] \subset I$, with the solution denoted $\phi_{x, u}: I \to \mathcal{D}$,
    \item $\mathcal{S}_h$ and $C_h$ are maximal (cannot be contained in a superset and containing correspondence)
\end{itemize}
Define $f_h: \mathrm{Graph}(C_h) \to \mathcal{D}$ for all pairs $(x, u) \in \mathrm{Graph}(C_h)$ as:
\begin{equation}
    f_h(x, u) = \phi_{x, u}(h).
\end{equation}

\begin{example}[Linear System]\label{ex:linear-system}
Let $\mathcal{D} = \R^d$. For matrices $A \in \R^{d \times d}$ and $B \in \R^{d \times m}$, suppose
$f_0(x) = Ax$ and $G(x) = B$ for all initial conditions $x \in \R^d$. Then $\mathcal{S}_h = \R^d$ and $C_h(x) = \R^m$ for all $x \in \R^d$. Moreover, $f_h$ can be expressed explicitly as a linear function in terms of the matrix exponential of $A$.
\end{example}

\begin{example}[Continuously Differentiable and Lipschitz Continuous Vector Fields]\label{ex:diff-and-lip}
Suppose $\mathcal{D} = \R^d$ and $f_0, g_1, \dots, g_m$ are continuously differentiable and (globally) Lipschitz continuous. For a control input $u \in \R^m$, since the $u$-section $F^u: \R^d \to \R^d$ is a linear combination of continuously differentiable and Lipschitz continuous vector fields, it is continuously differentiable and Lipschitz continuous. Therefore, for any initial condition $x \in \R^d$ the initial value problem characterized by initial condition $x$, control input $u$, and time interval $\R$ has a unique solution \cite[Theorem 3.1.3]{perko2013differential}. Therefore, $\mathcal{S}_h = \R^d$ and $ C_h(x) = \R^m$ for all initial conditions $x \in \R^d$. In general, $f_h$ does not have an closed-form representation.
\end{example}

\begin{assumption}\label{as:control-affine-well-posed}
The vector fields $f_0, g_1, \dots, g_m$ are locally Lipschitz continuous. Moreover, there exist subsets $\mathcal{S} \subseteq \mathcal{S}_h$ and $\mathcal{A} \subseteq \R^m$ and a correspondence $C: \mathcal{S} \to \mathcal{P}(\mathcal{A})$ such that:
\begin{itemize}
    \item For every initial condition $x \in \mathcal{S}$, the set of control inputs $C(x) \subseteq \mathcal{A}$ is nonempty with $C(x) \subseteq C_h(x)$,
    \item $\mathrm{Graph}(C)$ is a closed subset of $\mathcal{S} \times \mathcal{A}$,
    \item $f_h(\mathrm{Graph}(C)) \subseteq \mathcal{S}$.
\end{itemize}
\end{assumption}

Both Examples \ref{ex:linear-system} and \ref{ex:diff-and-lip} satisfy Assumption \ref{as:control-affine-well-posed} with $\mathcal{S} = \mathcal{S}_h$, $\mathcal{A} = \R^m$, and $C(x) = \mathcal{A}$ for all $x \in \mathcal{S}$. The assumption of local Lipschitz continuity in Assumption \ref{as:control-affine-well-posed} ensures that the restriction of $f_h$ to $\mathrm{Graph}(C)$ is continuous; this follows from \cite[Theorem 3.5]{khalil2002nonlinear} by treating control inputs as parameters.

\subsection{Robotic Systems}
Many robotic systems can be modeled as control affine systems, and we will demonstrate how to show that such systems can satisfy Assumption \ref{as:general}, establishing well-posedness of value iteration for a large class of practically relevant MPOPs. Examples of such systems include manipulators, automobiles, aircraft, and spacecraft \cite{murray1994mathematical, olfati2001nonlinear}.

\subsubsection{Dynamics}
An unconstrained robotic system with $n \in \N$ degrees of freedom is characterized by an $n$-dimensional $\mathcal{C}^2$ manifold $\mathcal{Q}$ called the configuration manifold. In this work, we consider open subsets $\mathcal{Q} \subseteq \R^n$ for which we can identify the tangent bundle $T\mathcal{Q}$ with $\mathcal{Q} \times \R^n$. The inertia matrix function $D: \mathcal{Q} \to \bb{S}^n_{++}$ characterizes the kinetic energy function $T: \mathcal{Q} \times \R^n \to \R_{++}$, defined as $T(q, \dot{q}) = \frac{1}{2}\dot{q}^\top D(q)\dot{q}$ for all configurations $q \in \mathcal{Q}$ and velocities $\dot{q} \in \R^n$. The assumption that $D$ takes positive definite values ensures that no configurations exist that admit arbitrarily large velocities without affecting the kinetic energy of the system. We require $D$ to be differentiable, allowing us to express the matrix-valued function $C: \mathcal{Q} \times \R^n \to \R^{n \times n}$ of Coriolis and centrifugal terms\footnote{The use of $C$ for this term is standard, but we will not use it in the same context as an action-admissibility correspondence.} as:
\begin{equation}
    (C(q, \dot{q}))_{ij} = \sum_{k = 1}^n \left(\frac{\partial}{\partial q_j} ( D(q) )_{ik} - \frac{1}{2} \frac{\partial}{\partial q_i} ( D(q) )_{jk} \right) \dot{q}_k,
\end{equation}
for all configurations $q = (q_1, \dots, q_n) \in \mathcal{Q}$, velocities $\dot{q} = (\dot{q}_1, \dots, \dot{q}_n) \in \R^n$, and indices $i, j \in \{ 1, \dots, n \}$ \cite[Equation 4.21]{murray1994mathematical}. We also require the potential energy function $U: \mathcal{Q} \to \R^n$ to be differentiable, and we denote external nonconservative forces and torques applied to the system with the vector-valued function $F_{\mathrm{ext}}: \mathcal{Q} \times \R^n \to \R^n$. Finally, if the system is controlled with $m \in \N$ actuators, then $B: \mathcal{Q} \to \R^{n \times m}$ denotes the actuation matrix function, converting control inputs to forces and torques. With $d = 2n$ and $\mathcal{D} = \mathcal{Q} \times \R^n$, we define $F: \mathcal{D} \times \R^m \to \R^d$ as:
\begin{align}
    F(x, u) &= \begin{bmatrix} \dot{q} \\ D(q)^{-1}( F_{\mathrm{ext}}(q, \dot{q}) - C(q, \dot{q})\dot{q} - \nabla U(q) ) \end{bmatrix}\nonumber\\
    &~\qquad + \begin{bmatrix} 0_{n \times m} \\ D(q)^{-1}B(q) \end{bmatrix}u,
\end{align}
for all $x = (q, \dot{q}) \in \mathcal{Q} \times \R^n$ and $u \in \R^m$.

\subsubsection{Regularity}
We now establish regularity conditions that enable us to define a valid transition function that can be used in an MDP. We first establish sufficient conditions such that every initial condition and control input under consideration correspond to initial value problems with unique solutions. We then establish sufficient conditions for these solutions to exist for all nonnegative time.

\begin{assumption}\label{as:robotic-well-posed}
The configuration manifold satisfies $\mathcal{Q} = \R^n$. The functions $D$, $C$, $\nabla U$, $F_{\mathrm{ext}}$, and $B$ are each locally Lipschitz continuous. There is a strictly positive lower bound $\lambda_{\mathrm{\min}}$ such that $\lambda_{\min} I_n \preceq D(q)$ for all configurations $q \in \R^n$. There is some closed set $G \subseteq \R^m$ such that for any control input $u \in G$ there are corresponding constants $c_0 \in \R_+$ and $c_1 \in \R_{++}$ such that:
\begin{equation}
    \| F_{\mathrm{ext}}(q, \dot{q}) - \nabla U(q) + B(q)u \|_2 \leq c_0 + c_1 \| \dot{q} \|_2,
\end{equation}
for all configurations $q \in \R^n$ and velocities $\dot{q} \in \R^n$.
\end{assumption}

\begin{remark}
If the derivative of $D$ is locally Lipschitz continuous, then so is $D$ itself as it is continuously differentiable, and so is $C$ as it is bilinear in the derivative of $D$ and the velocities. If $D$ is twice continuously differentiable, then the derivative of $D$ is locally Lipschitz continuous, implying $D$ and $C$ are as well. If $U$ is twice continuously differentiable, then $\nabla U$ is locally Lipschitz continuous.
\end{remark}

\begin{remark}
If $B$ is bounded by some $M \in \R_+$ and there are constants $d_0 \in \R_+$ and $d_1 \in \R_{++}$ such that:
\begin{equation}
    \| F_{\mathrm{ext}}(q, \dot{q}) - \nabla U(q) \|_2 \leq d_0 + d_1 \| \dot{q} \|_2,
\end{equation}
for all configurations $q \in \R^n$ and velocities $\dot{q} \in \R^n$, then for any $u \in \R^m$, we have:
\begin{equation}
    \| F_{\mathrm{ext}}(q, \dot{q}) - \nabla U(q) + B(q)u \|_2 \leq d_0 + M \| u \|_2 + d_1 \| \dot{q} \|_2,
\end{equation}
for all configurations $q \in \R^n$ and velocities $\dot{q} \in \R^n$. Therefore, choose $c_0 = d_0 + M \| u \|_2$ and $c_1 = d_1$, as well as $G = \R^m$.
\end{remark}

The proof of \cite[Theorem 9.8b]{rudin1976principles} shows that matrix inversion is locally Lipschitz continuous. Fixing a control input $u \in G$, the $u$-section $F^u$ is locally Lipschitz continuous as it is comprised of sums and products of locally Lipschitz continuous functions. Fix an initial configuration $q_0 \in \R^n$ and velocity $\dot{q}_0 \in \R^n$. There is a corresponding maximal open interval $I_{\max} \subseteq \R$ with $0 \in I_{\max}$ such that the initial value problem characterized by initial condition $(q_0, \dot{q}_0)$, control input $u$, and time interval $I_{\max}$ has a unique solution $\phi: I_{\max} \to \R^{2n}$ \cite[Theorem 2.4.1]{perko2013differential} (the proof of this theorem applies to locally Lipschitz continuous vector fields, not just continuously differentiable vector fields). Let $\psi: I_{\max} \to \R^n$ satisfy $\phi(t) = ( \psi(t), \dot{\psi}(t) )$ for all times $t \in I_{\max}$, and note that $\psi(0) = q_0$ and $\dot{\psi}(0) = \dot{q}_0$. Let $I_{\max}^+ = \R_+ \cap I_{\max}$. The kinetic energy satisfies:
\begin{align}
    &T(\psi(t), \dot{\psi}(t)) - T(q_0, \dot{q}_0)\nonumber\\
    &~ = \int_0^t \dot{\psi}(s)^\top ( F(\psi(s), \dot{\psi}(s)) - \nabla U(\psi(s)) + B(\psi(s))u ) ~\mathrm{d}s\nonumber\\
    &~\leq \int_0^t \| \dot{\psi}(s) \|_2 (c_0 + c_1 \| \dot{\psi}(s) \|_2) ~\mathrm{d}s,
\end{align}
for all times $t \in I_{\max}^+$. Since $T(q, \dot{q}) \geq \frac{1}{2} \lambda_{\min} \| \dot{q} \|_2^2$ for all configurations $q \in \R^n$ and velocities $\dot{q} \in \R^n$, we have:
\begin{align}
    &\left(\sqrt{\lambda_{\min}} \| \dot{\psi}(t) \|_2\right)^2 / 2 \leq \left(\sqrt{ 2T(q_0, \dot{q}_0) }\right)^2 / 2\nonumber\\
    &~\qquad\qquad + \int_0^t \frac{c_0 + c_1 \| \dot{\psi}(s) \|_2}{\sqrt{\lambda_{\min}}} (\sqrt{\lambda_{\min}} \| \dot{\psi}(s) \|_2) ~\mathrm{d}s,
\end{align}
for all times $t \in I_{\max}^+$. For any upper bound $t_f \in I_{\max}^+$ with $t_f > 0$, the functions $\phi_1, \phi_2: [0, t_f] \to \R_+$ defined as:
\begin{align}
    \phi_1(t) &= \sqrt{\lambda_{\min}} \| \dot{\psi}(t) \|_2, & \phi_2(t) = \frac{ c_0 + c_1 \| \dot{\psi}(t) \|_2 }{\sqrt{\lambda_{\min}}},
\end{align}
for all $t \in [0, t_f]$ are continuous and bounded, implying $\phi_2$ is absolutely integrable. Therefore, from \cite[Lemma 17]{ballard2000dynamics} (originally \cite[Lemma A.5]{brezis1973ope}), we have:
\begin{equation}
    \sqrt{ \lambda_{\min} } \| \dot{\psi}(t) \|_2 \leq \sqrt{ 2T(q_0, \dot{q}_0) } + \int_0^t \frac{c_0 + c_1 \| \dot{\psi}(s) \|_2}{\sqrt{\lambda_{\min}}} ~\mathrm{d}s,
\end{equation}
for all times $t \in [0, t_f]$. By the Gronwall-Bellman inequality \cite[Lemma A.1]{khalil2002nonlinear}, we have:
\begin{equation}\label{eq:gb-bound}
    \| \dot{\psi}(t) \|_2 \leq ( \sqrt{ 2T(q_0, \dot{q}_0) / \lambda_{\min} } + c_0 t / \lambda_{\min} ) e^{\frac{c_1}{\lambda_{\min}}t},
\end{equation}
for all times $t \in [0, t_f]$. Since $t_f$ was arbitrary, the bound \eqref{eq:gb-bound} holds for all times $t \in I_{\max}^+$.

We now use \eqref{eq:gb-bound} to show that $I_{\max}^+ = \R_+$. For contradiction, assume there is some time $t_f \in \R_{++}$ with $t_f \notin I_{\max}^+$. By upper bounding the right-hand side of \eqref{eq:gb-bound} by its value at the time $t_f$, we conclude that $\dot{\psi}$ is bounded on $I_{\max}^+$. Also:
\begin{align}
    &\| \psi(t) - q_0 \|_2 \leq \int_0^t \| \dot{\psi}(s) \|_2 ~\mathrm{d}s\nonumber\\
    &\qquad~\leq t_f \left( \sqrt{ 2T(q_0, \dot{q}_0) / \lambda_{\min} } + c_0 t_f / \lambda_{\min} \right) e^{\frac{c_1}{\lambda_{\min}}t_f},
\end{align}
for all times $t \in I_{\max}^+$; that is, $\psi$ is bounded on $I_{\max}^+$. Therefore, $\phi$ is bounded on $I_{\max}^+$, contradicting \cite[Theorem 2.4.3]{perko2013differential} (again, the proof of this theorem applies to locally Lipschitz continuous vector fields). This implies $I_{\max}^+ = \R_+$, meaning for any sample period $h \in \R_{++}$, there is an open interval $I \subseteq \R$ with $[0, h] \subset I$.

Thus, any robotic system satisfying Assumption \ref{as:robotic-well-posed} also satisfies Assumption \ref{as:control-affine-well-posed}, no matter which sample period $h$ is chosen.

\subsection{Sampled-Data Control}
\label{sec:sampled-data-control}
We now determine sufficient conditions for the existence of a continuous policy rendering a compact subset of the state space forward-invariant. This construction uses the methods of \cite{taylor2021sampled, taylor2022safety}, based on the original work \cite{nevsic1999sufficient} in sampled-data control.

A control affine system is full-state feedback linearizable if $\mathcal{D}$ is open and there exist a function $\Phi: \mathcal{D} \to \R^d$ that is a diffeomorphism between $\mathcal{D}$ and an open subset of $\R^d$, a continuous controller $k_{\mathrm{fbl}}: \R^d \times \R^m \to \R^m$ that accepts an auxiliary input, and a controllable pair $(A, B) \in \R^{d \times d} \times \R^{d \times m}$ such that:
\begin{equation}
    \frac{\partial \Phi}{\partial x} F(x, k_{\mathrm{fbl}}(x, v)) = A \Phi(x) + Bv,
\end{equation}
for all  $x \in \mathcal{D}$ and auxiliary control inputs $v \in \R^m$.

\begin{remark}
Any robotic system satisfying Assumption \ref{as:robotic-well-posed} with $n = m$ and $B(q)$ invertible for all configurations $q \in \R^n$ is full-state feedback linearizable. The feedback linearizing controller $k_{\mathrm{fbl}}: \R^d \times \R^n \to \R^n$ is defined as:
\begin{equation}
    k_{\mathrm{fbl}}(x, v) = B(q)^{-1} ( C(q, \dot{q})\dot{q} + \nabla U(q) - F_{\mathrm{ext}}(q, \dot{q}) + D(q)v ),
\end{equation}
for all $x = (q, \dot{q}) \in \R^n \times \R^n$ and $v \in \R^n$, and:
\begin{equation}
    F(x, k_{\mathrm{fbl}}(x, v)) = \begin{bmatrix} 0_{n \times n} & I_n \\ 0_{n \times n} & 0_{n \times n} \end{bmatrix} \begin{bmatrix} q \\ \dot{q} \end{bmatrix} + \begin{bmatrix} 0_{n \times n} \\ I_n \end{bmatrix}v,
\end{equation}
for all $x = (q, \dot{q}) \in \R^n \times \R^n$ and  $v \in \R^n$.
\end{remark}

Consider a full-state feedback linearizable control affine system, and suppose $0_d \in \mathcal{D}$ and $\Phi(0_d) = 0_d$. Consider any gain matrix $K \in \R^{m \times d}$ making $A - BK$ Hurwitz stable with all eigenvalues having strictly negative real parts. Suppose the system satisfies Assumption \ref{as:control-affine-well-posed} and $0_d \in \mathrm{int}(\mathcal{S})$. Additionally, suppose $\pi_0: \mathcal{S} \to \R^m$ defined as:
\begin{equation}
    \pi_0(x) = k_{\mathrm{fbl}}(x, -K\Phi(x)),
\end{equation}
for all states $x \in \mathcal{S}$ satisfies $\pi_0(x) \in C(x)$ for all states $x \in \mathcal{S}$. Since $k_{\mathrm{fbl}}$ and $\Phi$ are continuous, $\pi_0$ is continuous.

By \cite[Lemmas 3, 4]{taylor2021sampled}, there is a continuous (class $\mathcal{KL}$) function $\beta: \R_+ \times \R_+ \to \R_+$ with monotonically increasing $s$-sections $\beta^s$ for all $s \in \R_+$ and monotonically nonincreasing $r$-sections $\beta_r$ for all $r \in \R_+$, as well as some $\overline{R} \in \R_{++}$ and a bounded open set $N \subset \R^d$ with $0_d \in N$ and:
\begin{equation}
    \{ x \in \R^d: \| x \|_2 \leq \sup_{x' \in N} \beta(\| x' \|_2, 0) + \overline{R} \} \subseteq \mathcal{S}.
\end{equation}
Additionally, for any $R \leq \overline{R}$, if the sample period is sufficiently small, we have:
\begin{align}
    \| (f_h)_{\pi_0}^t(x) \|_2 &\leq \beta( \| x \|_2, th ) + R \leq \sup_{x' \in N} \beta(\| x' \|_2, 0) + \overline{R},
\end{align}
for all $x \in N$ and $t \in \Z_+$. Let $\mathcal{R} \subseteq \mathcal{S}$ denote the set of all states reachable from $N$ when following $\pi_0$, defined as $\mathcal{R} = \{ (f_h)_{\pi_0}^t(x) \in \mathcal{S}: x \in N, t \in \Z_+ \}$,
and note that $(f_h)_{\pi_0}(\mathcal{R}) \subseteq \mathcal{R}$ and $\mathrm{cl}(\mathcal{R}) \subseteq \mathcal{S}$. Since $(f_h)_{\pi_0}(\mathrm{cl}(\mathcal{R})) \subseteq \mathrm{cl}(\mathcal{R}) \subseteq \mathcal{S}$, we can set $\mathcal{S}_0 = \mathrm{cl}(\mathcal{R})$. The set $\mathcal{S}_0 \subseteq \mathcal{S}$ is closed and bounded since:
\begin{equation}
    \sup_{x \in N} \beta(\|x\|_2, 0) + \overline{R} = \beta( \sup_{x \in N} \| x \|_2, 0 ) + \overline{R} < \infty.
\end{equation}
Therefore, $\mathcal{S}_0$ is a compact set rendered forward-invariant by the continuous policy $\pi_0$, implying the system satisfies Assumption \ref{as:general}.

Finally, we can construct a compact subset of the state space in a similar manner by appealing to \cite[Theorem 3]{taylor2022safety}, employing the notion of practical safety. In this case, if a compact $0$-superlevel set of a family of sampled-data control barrier functions \cite[Definition 7]{taylor2022safety} can be rendered forward-invariant when the transition map is approximated by an appropriate Euler/Runge-Kutta scheme, then an enlarged (but bounded) set can be rendered forward-invariant for the exact transition map when the sample period is sufficiently small. The enlargement can be controlled by decreasing the sample period.

\section{Future Work}
There are many directions for future work. First, our results can be viewed as ``existence'' results, and set the stage for studying computational aspects of policy optimization in metric MDPs. Second, our results are applicable to infinite-dimensional state and action spaces, such as those encountered in the control of systems governed by partial differential equations. This setting provides new challenges, such as how to choose appropriate compact subsets of the state space, as well as further computational considerations. Third, and it would be interesting to also consider stochastic dynamics. Finally, it is important to study the interaction between policy optimization and model learning when a model must be derived from data.

\bibliographystyle{IEEEtran} 
\bibliography{refs}

\begin{IEEEbiography}[{\includegraphics[width=1in,height=1.25in,clip,keepaspectratio]{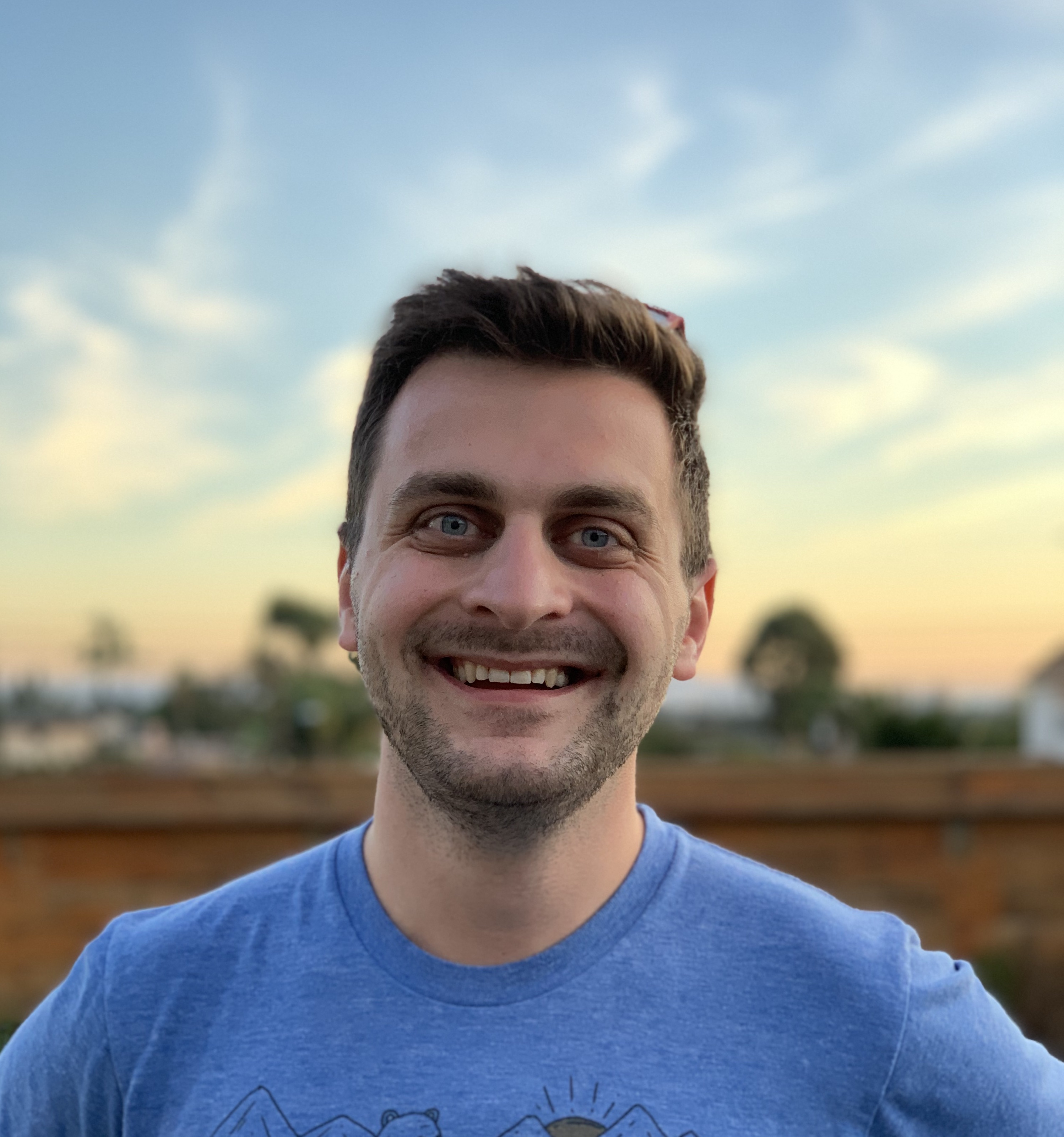}}]{Victor D. Dorobantu} was born in Stockholm, Sweden in 1995. He received the B.S. degree in CS from Cornell University in 2017, and is currently pursuing the Ph.D. degree in computing and mathematical sciences at the California Institute of Technology (Caltech). His research interests are reinforcement learning theory, geometry, and nonlinear control.
\end{IEEEbiography}

\begin{IEEEbiography}[{\includegraphics[width=1in,height=1.25in,clip,keepaspectratio]{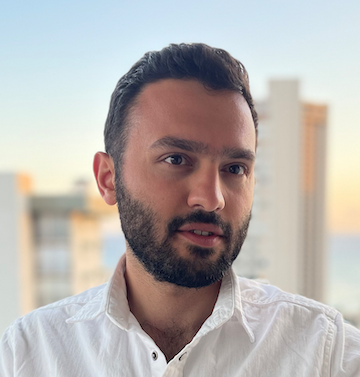}}]{Kamyar Azizzadenesheli} was born in Iran, in 1992. He received the B.S. degree in EE from the Sharif University of Technology, in 2014 and the M.S. and Ph.D. degrees in EE from the University of California, Irvine, CA, in 2019. From 2019 to 2020, he was a postdoctoral scholar at the California Institute of Technology. Since 2020, he has been an Assistant Professor with the Computer Science Department, Purdue University. He is the author of a digital book and more than 38 articles. His research interests include machine learning, from theory to practice.
\end{IEEEbiography}

\begin{IEEEbiography}[{\includegraphics[width=1in,height=1.25in,clip,keepaspectratio]{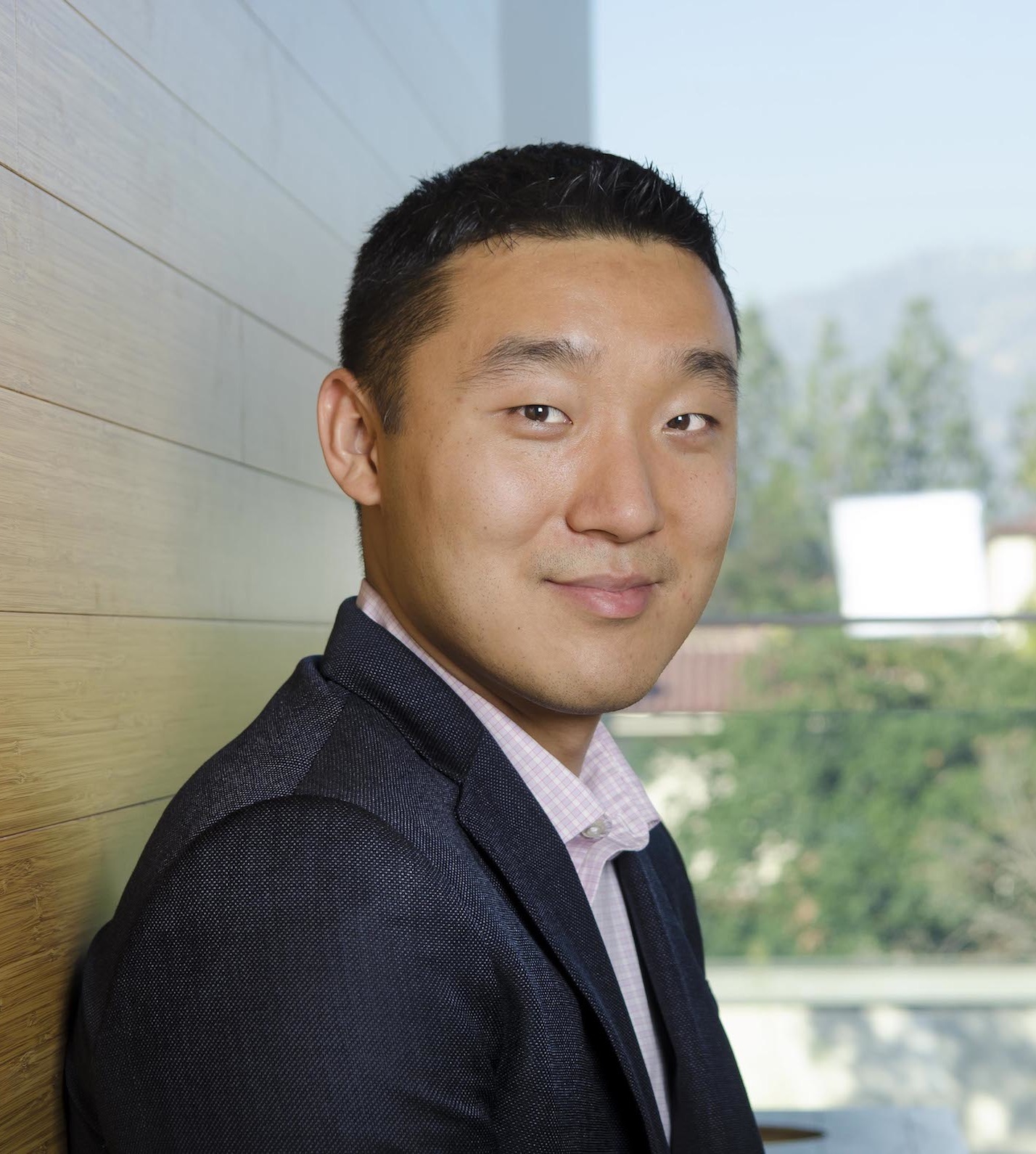}}]{Yisong Yue} was born in Beijing in 1983.  He received a B.S. in CS from the University of Illinois at Urbana-Champaign in 2005, and a Ph.D in CS from Cornell Universitry in 2010.  From 2010-2013, he was a postdoctoral scholar at Carnegie Mellon University.  From 2013-2014, he was a research scientist at Disney.  Since 2014, he has been on the faculty at the California Institute of Technology, where he is currently a full professor.  His research interests are in machine learning.
\end{IEEEbiography}

\section*{Appendix}

We offer alternative, direct proofs of Theorem \ref{thm:well-posed} and Lemma \ref{lem:relax-assumption}. We proceed as follows:
\begin{inparaenum}[1)]
    \item Introduce additional notation, conventions, and definitions,
    \item Show the measurability and boundedness of value functions under measurable policies,
    \item Introduce Bellman and optimal Bellman operators,
    \item Prove a selection theorem for bounded and upper semicontinuous functions on graphs of correspondences,
    \item Connect these results to proofs of Theorem \ref{thm:well-posed} and Lemma \ref{lem:relax-assumption}.
\end{inparaenum}

\subsection{Additional Notation, Conventions, and Definitions}
For a nonempty metric spaces $X$ and $Y$, the set of bounded measurable functions from $X$ to $\R$ is denoted $\mathcal{L}_b^0(X)$, which is a Banach space when equipped with the norm $\|\cdot\|_{\sup}: \mathcal{L}_b^0(X) \to \R$, defined as:
\begin{equation}
    \| f \|_{\sup} = \sup_{x \in X} | f(x) |,
\end{equation}
for all bounded measurable functions $f: X \to \R$. The set of bounded upper semicontinuous functions from $X$ to $\R$ is denoted $\mathcal{C}_b^u(X)$, the set of bounded continuous functions from $X$ to $\R$ is denoted $\mathcal{C}_b^0(X)$, and note that $\mathcal{C}_b^0(X) \subseteq \mathcal{C}_b^u(X) \subseteq \mathcal{L}_b^0(X)$.

For another nonempty metric space $Y$, consider a correspondence $C: X \to \mathcal{P}(Y)$. $C$ is measurable if the lower preimage of each closed set is a measurable set, and $C$ is weakly measurable if the lower preimage of each open set is a measurable set. If $C$ is measurable, then it is weakly measurable \cite[Lemma 18.2]{aliprantis2006infinite}.

The product $X \times Y$ is a metric space, with a topology generated by the basis $\{ U \times V: U \subseteq X ~\mathrm{and}~ V \subseteq Y ~\mathrm{are~open} \}$, and, if $X$ and $Y$ are separable, a $\sigma$-algebra generated by the collection $\{ A \times B: A \subseteq X ~\mathrm{and}~ B \subseteq Y ~\mathrm{are~measurable~sets} \}$ \cite[Appendix M]{billingsley2013convergence}.

For another nonempty metric space $Z$, consider a function $f: X \times Y \to Z$. If $f_x$ is continuous for each $x \in X$ and $f^y$ is measurable for each $y \in Y$, then $f$ is called a Carathéodory function; if $Y$ is also separable, then $f$ is measurable \cite[Lemma 4.51]{aliprantis2006infinite}.

In what follows, consider an MPOP characterized by an MDP $(\mathcal{S}, \mathcal{A}, C, f, r, \gamma)$ and a policy class $\Pi$ satisfying Assumption \ref{as:well-posed}.

\subsection{Measurability and Boundedness}
To establish the measurability of the value function under a policy $\pi \in \Pi$, we begin by establishing measurability of the corresponding closed-loop transition map $f_\pi$ and single-step reward function $r_\pi$. To this end, we define $z: \mathcal{S} \to \mathcal{S} \times \mathcal{A}$ as:
\begin{equation}
    z(s) = (s, \pi(s)),
\end{equation}
for all states $s \in \mathcal{S}$. Consider measurable sets $A \subseteq \mathcal{S}$ and $B \subseteq \mathcal{A}$; the preimage of the product $A \times B$ under $z$ is $z^{-1}(A \times B) = \{ s \in \mathcal{S}: s \in A, \pi(s) \in B \} = A \cap \pi^{-1}(B)$, which is a measurable set since $\pi$ is a measurable function. Since measurable products generate the $\sigma$-algebra on $\mathcal{S} \times \mathcal{A}$,  $z$ is measurable, implying the compositions $f_\pi = f \circ z$ and $r_\pi = r \circ z$ are measurable functions. It follows that for any $t \in \Z_+$, the $t$-iterated composition $f_\pi^t$ is a measurable function, as is the composition $r_\pi \circ f_\pi^t$. Note that $\mathcal{S}_\pi = \mathcal{S}$ since $r_\pi$ is bounded (by the same bound on $r$); therefore, the corresponding value function $V_\pi: \mathcal{S} \to \R$ is well-defined on all of $\mathcal{S}$. Moreover, $V_\pi$ is measurable since it is the pointwise limit of a sequence of partial sums of measurable functions. Finally, $V_\pi$ is bounded by $(1 - \gamma)^{-1}$ times the bound on $r$, so $V_\pi \in \mathcal{L}_b^0(\mathcal{S})$.

\subsection{Bellman Operators}
The Bellman operator $\mathcal{T}_\pi: \mathcal{L}_b^0(\mathcal{S}) \to \mathcal{L}_b^0(\mathcal{S})$ under policy $\pi$ generalizes the implicit structure of the value function $V_\pi$, satisfying:
\begin{equation}
    [\mathcal{T}_\pi(g)](s) = r_\pi(s) + \gamma g(f_\pi(s)),
\end{equation}
for all states $s \in \mathcal{S}$, for all bounded measurable functions $g: \mathcal{S} \to \R$. Indeed, for any bounded measurable function $g: \mathcal{S} \to \R$, the function $\mathcal{T}_\pi(g) = r_\pi + \gamma g \circ f_\pi$ is bounded and measurable. The Bellman operator $\mathcal{T}_\pi$ is a $\gamma$-contraction \cite[\S 1.4.1]{bertsekas2011dynamic}, so by the contraction mapping principle \cite[Theorem 3.48]{aliprantis2006infinite}, $V_\pi$ is the only bounded measurable function with $\mathcal{T}_\pi(V_\pi) = V_\pi$.

We use Bellman operators to characterize the optimal Bellman operator, which is commonly used in value iteration. Since each value function corresponding to a policy in $\Pi$ is bounded by $(1 - \gamma)^{-1}$ times the bound on $r$, the optimal value function $V^*$ is well-defined and bounded by the same bound. The optimal Bellman operator $\mathcal{T}: \mathcal{C}_b^u(\mathcal{S}) \to \mathcal{C}_b^u(\mathcal{S})$ satisfies:
\begin{equation}\label{eq:opt-bellman}
    [\mathcal{T}(g)](s) = \sup_{\pi \in \Pi}{ [\mathcal{T}_\pi(g)](s)} = \sup_{a \in C(s)}{( r(s, a) + \gamma g(f(s, a)) )},
\end{equation}
for all states $s \in \mathcal{S}$, for all bounded upper semicontinuous functions $g: \mathcal{S} \to \R$. We cannot define $\mathcal{T}$ on all of $\mathcal{L}_b^0(\mathcal{S})$, and ensuring that $\mathcal{T}$ is well-defined (the suprema are finite and equal) and has codomain $\mathcal{C}_b^u(\mathcal{S})$ is the subject of the following lemma.

\begin{lemma}\label{lem:usc-measurable-selection}
Suppose an MDP $(\mathcal{S}, \mathcal{A}, C, f, r, \gamma)$ and policy class $\Pi$ satisfy Assumption \ref{as:well-posed}. Let $\varphi: \mathrm{Graph}(C) \to \R$ be upper semicontinuous and bounded. The function $g_\varphi: \mathcal{S} \to \R$ specified as:
\begin{equation}
    g_\varphi(s) = \max_{a \in C(s)} \varphi(s, a),
\end{equation}
for all states $s \in \mathcal{S}$ is well-defined, upper semicontinuous, and bounded. Moreover, there is a policy $\pi_\varphi \in \Pi$ satisfying:
\begin{equation}
    \varphi(s, \pi_\varphi(s)) = g_\varphi(s),
\end{equation}
for all states $s \in \mathcal{S}$.
\end{lemma}

\begin{proof}
The function $g_\varphi$ is well-defined and upper semicontinuous by \cite[Lemma 17.30]{aliprantis2006infinite} and is bounded by the bound on $\varphi$.

Define the correspondence $C_\varphi: \mathcal{S} \to \mathcal{P}(\mathcal{A})$ as:
\begin{equation}
    C_\varphi(s) = \argmax_{a \in C(s)} \varphi(s, a) \subseteq C(s),
\end{equation}
for all states $s \in \mathcal{S}$. For any state $s \in \mathcal{S}$, since $g_\varphi$ is well-defined, we have $C_\varphi(s) \neq \emptyset$. The singleton set $\{ s \}$ and the set of admissible actions $C(s)$ are both compact, implying the product $\{ s \} \times C(s)$ is compact. Therefore, the set of maximizing state-action pairs $\argmax_{(s', a') \in \{ s \} \times C(s)} \varphi(s', a')$ is nonempty and compact \cite[Theorem 2.43]{aliprantis2006infinite}. Since:
\begin{equation}
    C_\varphi(s) = p_{\mathcal{A}}\left( \argmax_{(s', a') \in \{ s \} \times C(s)} \varphi(s', a') \right),
\end{equation}
where $p_{\mathcal{A}}: \mathcal{S} \times \mathcal{A} \to \mathcal{A}$ is the continuous function projecting $\mathcal{S} \times \mathcal{A}$ onto $\mathcal{A}$, the set of maximizing actions $C_\varphi(s)$ is compact.

The remainder of the proof is a modification of the proof of \cite[Theorem 2.5.A]{dynkin1979controlled}. Since $C$ is upper hemicontinuous, it is measurable. Therefore, $C$ is weakly measurable \cite[Lemma 18.2]{aliprantis2006infinite}. Let $\rho_{\mathcal{A}}: \mathcal{A} \times \mathcal{A} \to \R_+$ denote the metric on $\mathcal{A}$, and define the distance function $\rho_C: \mathcal{S} \times \mathcal{A} \to \R_+$ as:
\begin{equation}\label{eq:distance-function}
    \rho_C(s, a) = \inf_{a' \in C(s)} \rho_{\mathcal{A}}(a, a'),
\end{equation}
for all states $s \in \mathcal{S}$ and actions $a \in \mathcal{A}$. Since $C$ is weakly measurable, $\rho_C$ is a Carathéodory function \cite[Theorem 18.5]{aliprantis2006infinite}; that is, for every state $s \in \mathcal{S}$, the $s$-section $(\rho_C)_s$ is continuous, and for every action $a \in \mathcal{A}$, the $a$-section $(\rho_C)^a$ is a measurable function. 

Since $C$ is upper hemicontinuous and compact valued and $\mathcal{A}$ is a Hausdorff space (as it is a metric space), $\mathrm{Graph}(C)$ is a closed subset of $\mathcal{S} \times \mathcal{A}$ \cite[Theorem 17.10]{aliprantis2006infinite}. As discussed in \cite{295497} (and in a similar manner to the proof of \cite[Corollary 2]{askoura2008extension}), let $M \in \R_+$ denote the bound on $\varphi$, and define the extension $\bar{\varphi}: \mathcal{S} \times \mathcal{A} \to \R$ as:
\begin{equation}
    \bar{\varphi}(s, a) = \begin{cases}
        \varphi(s, a) & (s, a) \in \mathrm{Graph}(C),\\
        -M & (s, a) \in (\mathcal{S} \times \mathcal{A}) \setminus \mathrm{Graph}(C),
    \end{cases}
\end{equation}
for all states $s \in \mathcal{S}$ and actions $a \in \mathcal{A}$. To see that $\bar{\varphi}$ is upper semicontinuous, consider any real number $c \in \R$. If $c > -M$, then $\bar{\varphi}^{-1}([c, \infty)) = \varphi^{-1}([c, \infty))$. Since $\varphi$ is upper semicontinuous, there is a closed set $F_c \subseteq \mathcal{S} \times \mathcal{A}$ satisfying $\varphi^{-1}([c, \infty)) = \mathrm{Graph}(C) \cap F_c$, making $\bar{\varphi}^{-1}([c, \infty))$ a closed subset of $\mathcal{S} \times \mathcal{A}$. Otherwise, if $c \leq -M$, then $\bar{\varphi}^{-1}([c, \infty)) = \mathcal{S} \times \mathcal{A}$, which is a closed set.

Consider a monotonically nonincreasing sequence of bounded real-valued continuous functions $\{ h_n \in \mathcal{C}^0(\mathcal{S} \times \mathcal{A}): n \in \N \}$ converging pointwise to $\bar{\varphi}$; such a sequence is guaranteed to exist \cite[\S 2.4]{dynkin1979controlled}.
For each $n \in \N$, define the correspondence $C_n: \mathcal{S} \to \mathcal{P}(\mathcal{A})$ as:
\begin{align}
    C_n(s) &= \{ a \in \mathcal{A}: \rho_C(s, a) < 1/n \nonumber\\
    &~\qquad\qquad \mathrm{and}~ g_\varphi(s) < h_n(s, a) + 1/n \},
\end{align}
for all states $s \in \mathcal{S}$. Now, for any state $s \in \mathcal{S}$, recall that the $s$-section $(\rho_C)_s$ is continuous. For any $n \in \N$, since $h_n$ is continuous, the corresponding $s$-section $(h_n)_s$ is continuous. Thus, $C_n(s)$ is an open subset of $\mathcal{A}$, and can be represented as:
\begin{align}
    C_n(s) &= \left\{ a \in \mathcal{A}: (\rho_C)_s(a) \in \left(-\infty, 1/n\right) \right\}\nonumber\\
    &~\qquad \cap \left\{ a \in \mathcal{A}: (h_n)_s(a) \in \left( g_\varphi(s) - 1/n, \infty \right) \right\}\nonumber\\
    &= ((\rho_C)_s)^{-1}\left(\left(-\infty, 1/n\right)\right)\nonumber\\
    &~\qquad \cap ((h_n)_s)^{-1}\left(\left(g_\varphi(s) - 1/n, \infty\right)\right).
\end{align}
Since the sequence of functions $\{ h_n: n \in \N \}$ is monotonically nonincreasing, these correspondences satisfy $C_{n + 1}(s) \subseteq C_n(s)$ for all states $s \in \mathcal{S}$ and $n \in \N$. For any state $s \in \mathcal{S}$ and any action $a \in C_\varphi(s)$, we have $\rho_C(s, a) = 0 < 1/n$ and $g_\varphi(s) = \varphi(s, a) \leq h_n(s, a) < h_n(s, a) + 1/n$ for all $n \in \N$. This implies $C_\varphi(s) \subseteq \bigcap_{n = 1}^\infty C_n(s)$ for all states $s \in \mathcal{S}$. 

Next, note that $g_\varphi$ is a measurable function as it is bounded and upper semicontinuous. For any action $a \in \mathcal{A}$, recall that the $a$-section $(\rho_C)^a$ is a measurable function. Additionally, for any $n \in \N$, since $h_n$ is continuous, the corresponding $a$-section $(h_n)^a$ is a measurable function (as it is continuous). This means the following set is measurable:
\begin{align}
    &\{ s \in \mathcal{S}: a \in C_n(s) \}\nonumber\\
    &~= \left\{ s \in \mathcal{S}: (\rho_C)^a(s) \in \left(-\infty, 1/n\right) \right\}\nonumber\\
    &~\qquad\qquad\cap \left\{ s \in \mathcal{S}: g_\varphi(s) - ((h_n)^a)(s) \in \left(-\infty, 1/n\right) \right\}\nonumber\\
    &~= ((\rho_C)^a)^{-1}\left(\left(-\infty, 1/n\right)\right)\nonumber\\
    &~\qquad\qquad \cap (g_\varphi - (h_n)^a)^{-1}\left(\left(-\infty, 1/n\right)\right).
\end{align}

Lastly, for any state $s \in \mathcal{S}$, consider a sequence of actions $\{ a_n \in \mathcal{A}: n \in \N \}$ satisfying $a_n \in C_n(s)$ for all $n \in \N$. For each $n \in \N$, since $\rho_C(s, a_n) < 1 / n$, there is an action $a_n' \in C(s)$ satisfying $\rho_{\mathcal{A}}(a_n', a_n) < 2 / n$. Since $C(s)$ is compact, the sequence of actions $\{ a_n' \in C(s): n \in \N \}$ has a limit point $a^* \in C(s)$. That is, there is a monotonically increasing sequence $\{ n_k \in \N: k \in \N \}$ such that the subsequence $\{ a_{n_k}' \in C(s): k \in \N \}$ converges to $a^*$. Since $\rho_{\mathcal{A}}(a_{n_k}', a_{n_k}) < 2 / n_k \leq 2 / k$ for all $k \in \N$, the subsequence $\{ a_{n_k} \in \mathcal{A}: k \in \N \}$ also converges to $a^*$. For any $k, l \in \N$ with $k > l$, we have:
\begin{equation}
    h_{n_l}(s, a_{n_k}) \geq h_{n_k}(s, a_{n_k}) > g_\varphi(s) - 1/n_k \geq g_\varphi(s) - 1 / k,
\end{equation}
and since $h_{n_l}$ is continuous, this means $\lim_{k \to \infty} h_{n_l}(s, a_{n_k}) = h_{n_l}(s, a^*) \geq g_\varphi(s)$. Since the sequence of functions $\{ h_n: n \in \N \}$ converges pointwise to $\bar{\varphi}$ from above, we have $\lim_{l \to \infty} h_{n_l}(s, a^*) = \bar{\varphi}(s, a^*) = \varphi(s, a^*) \geq g_\varphi(s)$, implying $a^* \in C_\varphi(s)$. That is, the sequence of actions $\{ a_n: n \in \N \}$ has a limit point in $C_\varphi(s)$. By the measurability criterion in \cite[\S 2.6]{dynkin1979controlled}, the correspondence $C_\varphi$ satisfies the assumptions of \cite[Theorem 2.6.B]{dynkin1979controlled}; therefore, there exists a measurable selection $\pi_\varphi: \mathcal{S} \to \mathcal{A}$ of $C_\varphi$, implying $\pi_\varphi \in \Pi$ and $\varphi(s, \pi_\varphi(s)) = g_\varphi(s)$ for all states $s \in \mathcal{S}$.
\end{proof}

Equipped with Lemma \ref{lem:usc-measurable-selection}, we only need to show that the suprema in \eqref{eq:opt-bellman} are equal. Fix a bounded measurable function $g: \mathcal{S} \to \R$ and a state $s \in \mathcal{S}$. For any policy $\pi \in \Pi$, we have:
\begin{align}\label{eq:bellman-bound-for-policy}
    [\mathcal{T}_\pi(g)](s) &= r(s, \pi(s)) + \gamma g(f(s, \pi(s)))\nonumber\\
    &\leq \sup_{a \in C(s)} {(r(s, a) + \gamma g(f(s, a)))}.
\end{align}
Conversely, for any action $a \in C(s)$, we have:
\begin{align}\label{eq:bellman-bound-for-action}
    &r(s, a) + \gamma g(f(s, a)) = r_{\pi_{s, a}}(s) + \gamma g(f_{\pi_{s, a}}(s)) \nonumber\\
    &~\qquad= [\mathcal{T}_{\pi_{s, a}}(g)](s) \leq \sup_{\pi \in \Pi}{ [\mathcal{T}_\pi(g)](s) },
\end{align}
where $\pi_{s, a} \in \Pi$ satisfies $\pi_{s, a}(s) = a$. Considering all policies in \eqref{eq:bellman-bound-for-policy} and all actions in \eqref{eq:bellman-bound-for-action} establishes the equality of the suprema. Since $r$ and $g$ are bounded and upper semicontinuous and $f$ is continuous, the function $r + \gamma g \circ f$ is bounded and upper semicontinuous, so applying Lemma \ref{lem:usc-measurable-selection} to this function allows us to conclude that $\mathcal{T}(g) \in \mathcal{C}_b^u(\mathcal{S})$.

To show that the MPOP is well-posed, we show that the optimal value function $V^*$ is the fixed point of $\mathcal{T}$; that is $\mathcal{T}(V^*) = V^*$. However, though $\mathcal{T}$ is a $\gamma$-contraction \cite[\S 1.4.1]{bertsekas2011dynamic}, we cannot immediately apply the contraction mapping principle \cite[Theorem 3.48]{aliprantis2006infinite} to show the existence of a unique fixed point of $\mathcal{T}$ since $\mathcal{C}_b^u(\mathcal{S})$ is not a closed subset of $\mathcal{L}_b^0(\mathcal{S})$. 

Instead, we consider the sequence of bounded upper semicontinuous function $\{ V_n \in \mathcal{C}_b^u(\mathcal{S}): n \in \N \}$ generated by value iteration when the initial guess is chosen as the constant function $V_0 = M / (1 - \gamma)$, where $M \in \R_+$ denotes the bound on $r$. To verify that the functions in this sequence are bounded and upper semicontinuous, note that the initial guess is bounded and upper semicontinuous, and by induction, the value iteration update can be written as:
\begin{equation}
    V_{n + 1} = \mathcal{T}(V_n),
\end{equation}
for all $n \in \Z_+$. The sequence of functions is also monotonically nonincreasing. To see this, note that:
\begin{align}
    V_1(s) &= [\mathcal{T}(V_0)](s) = \sup_{a \in C(s)} (r(s, a) + \gamma V_0(f(s, a)))\nonumber\\
    &\leq M + \gamma M / (1 - \gamma) = M / (1 - \gamma) \leq V_0(s),
\end{align}
for all states $s \in \mathcal{S}$. Since $\mathcal{T}$ preserves ordering \cite{bertsekas2011dynamic}, if for some $n \in \Z_+$ we have $V_{n + 1}(s) \leq V_n(s)$ for all states $s \in \mathcal{S}$, then $[\mathcal{T}(V_{n + 1})](s) \leq [\mathcal{T}(V_n)](s)$ for all states $s \in \mathcal{S}$. Equivalently, $V_{n + 2}(s) \leq V_{n + 1}(s)$ for all states $s \in \mathcal{S}$. By induction, it follows that the sequence of functions is monotonically nonincreasing. Finally, the sequence of functions is bounded by $M / (1 - \gamma)$. By definition, $V_0$ satisfies this bound, and since the sequence of functions is nonincreasing, it is upper bounded by $M / (1 - \gamma)$. To show the lower bound of $-M / (1 - \gamma)$, note that:
\begin{equation}
    V_1(s) \geq -M + \gamma V_0(f(s, a)) \geq -M - \frac{\gamma}{1 - \gamma}M = -\frac{M}{1 - \gamma},
\end{equation}
for all states $s \in \mathcal{S}$. Similarly, if for some $n \in \Z_+$ we have $V_n(s) \geq -M / (1 - \gamma)$ for all states $s \in \mathcal{S}$, then:
\begin{equation}
    V_{n + 1}(s) \geq -M + \gamma V_n(f(s, a)) \geq -M - \frac{\gamma}{1 - \gamma} = -\frac{M}{1 - \gamma},
\end{equation}
for all states $s \in \mathcal{S}$. Therefore, the sequence of functions is bounded below by induction.

The pointwise limit of a monotonically nonincreasing sequence of bounded upper semicontinuous functions is upper semicontinuous \cite[\S 2.4]{dynkin1979controlled}, and since the sequence is bounded, so is the limit. Denote the limit by $g^* \in \mathcal{C}_b^u(\mathcal{S})$. To show that $g^*$ is a fixed point of $\mathcal{T}$, note that:
\begin{align}
    \| \mathcal{T}(g^*) - g^* \|_{\sup} &\leq \| \mathcal{T}(g^*) - V_{n + 1} \|_{\sup} + \| V_{n + 1} - g^* \|_{\sup}\nonumber\\
    &= \| \mathcal{T}(g^*) - \mathcal{T}(V_n) \|_{\sup} + \| V_{n + 1} - g^* \|_{\sup}\nonumber\\
    &\leq \gamma \| g^* - V_n \|_{\sup} + \| V_{n + 1} - g^* \|_{\sup},
\end{align}
for all $n \in \Z_+$. By choosing $n$ sufficiently large, we make $\| \mathcal{T}(g^*) - g^* \|_{\sup}$ arbitrarily small, implying $\| \mathcal{T}(g^*) - g^* \|_{\sup} = 0$, or $\mathcal{T}(g^*) = g^*$. Now, consider any other sequence of bounded upper semicontinuous functions $\{ W_n \in \mathcal{C}_b^u(\mathcal{S}): n \in \N \}$ generated by value iteration from a bounded and upper semicontinuous initial guess $W_0 \in \mathcal{C}_b^u(\mathcal{S})$. Since $\mathcal{T}$ is a $\gamma$-contraction, we have:
\begin{equation}
    \| W_{n + 1} - g^* \|_{\sup} = \| \mathcal{T}(W_n) - \mathcal{T}(g^*) \|_{\sup} \leq \gamma \| W_n - g^* \|_{\sup},
\end{equation}
for all $n \in \Z_+$, or:
\begin{equation}
    \| W_n - g^* \|_{\sup} \leq \gamma^n \| W_0 - g^* \|_{\sup},
\end{equation}
for all $n \in \Z_+$. This implies $\lim_{n \to \infty} W_n = g^*$.

We prove Theorem \ref{thm:well-posed} by showing that $V^* = g^*$.

\begin{proof}[Proof of Theorem \ref{thm:well-posed}]
First, for any policy $\pi \in \Pi$, we have:
\begin{align}
    g^*(s) &= \sup_{a \in C(s)} {(r(s, a) + \gamma g^*(f(s, a)))}\nonumber\\
    &\geq r_\pi(s) + \gamma g^*(f_\pi(s)),
\end{align}
for all states $s \in \mathcal{S}$. For any $T \in \N$, assume: 
\begin{equation}\label{eq:opt-bellman-fixed-point-lower-bound}
    g^*(s) \geq \sum_{t = 0}^{T - 1} \gamma^t r_\pi(f_\pi^t(s)) + \gamma^T g^*(f_\pi^T(s)),
\end{equation}
for all states $s \in \mathcal{S}$. Then:
\begin{align}
    &g^*(s)\nonumber\\
    &~\geq \sum_{t = 0}^{T - 1} \gamma^t r_\pi(f_\pi^t(s)) + \gamma^T ( r_\pi(f_\pi^T(s)) + \gamma g^*(f_\pi(f_\pi^T(s))) )\nonumber\\
    &~= \sum_{t = 0}^T \gamma^t r_\pi(f_\pi^t(s)) + \gamma^{T + 1} g^*(f_\pi^{T + 1}(s)),
\end{align}
for all states $s \in \mathcal{S}$. By induction, \eqref{eq:opt-bellman-fixed-point-lower-bound} holds for all $T \in \N$, and since $g^*$ is bounded, we have:
\begin{equation}
    g^*(s) \geq \sum_{t = 0}^\infty \gamma^t r_\pi(f_\pi^t(s)) = V_\pi(s),
\end{equation}
for all states $s \in \mathcal{S}$. Taking the supremum over all policies, we obtain $g^*(s) \geq \sup_{\pi \in \Pi} V_\pi(s) = V^*(s)$ for all states $s \in \mathcal{S}$.

We show the reverse inequality by following a modification of the proof of \cite[Theorem 6.3.1]{puterman2014markov}. Consider the sequence of bounded upper semicontinuous functions $\{ V_n \in \mathcal{C}_b^u(\mathcal{S}): n \in \N \}$ generated by value iteration from an arbitrary initial guess $V_0 \in \mathcal{C}_b^u(\mathcal{S})$. Since $\lim_{n \to \infty} V_n = g^*$, for any $\varepsilon \in \R_{++}$, there is a corresponding $N_\varepsilon \in \N$ such that:
\begin{equation}
    \| V_n - g^* \|_{\sup} < \frac{1 - \gamma}{1 + \gamma} \varepsilon,
\end{equation}
for all $n \in \N$ with $n \geq N_{\varepsilon}$. By Lemma \ref{lem:usc-measurable-selection}, there is a policy $\pi_\varepsilon \in \Pi$ satisfying:
\begin{align}\label{eq:vi-greedy}
    [\mathcal{T}_{\pi_\varepsilon}(V_{N_\varepsilon + 1})](s) &= r(s, \pi_\varepsilon(s)) + \gamma V_{N_\varepsilon + 1}(f(s, \pi_\varepsilon(s)))\nonumber\\
    &= \max_{a \in C(s)}{(r(s, a) + \gamma V_{N_\varepsilon + 1}(f(s, a)))}\nonumber\\
    &= [\mathcal{T}(V_{N_\varepsilon + 1})](s),
\end{align}
for all states $s \in \mathcal{S}$. Therefore:
\begin{align}
    &\| V_{\pi_\varepsilon} - V_{N_\varepsilon + 1} \|_{\sup} \nonumber\\
    &~\leq \| V_{\pi_\varepsilon} - \mathcal{T}(V_{N_\varepsilon + 1}) \|_{\sup} + \| \mathcal{T}(V_{N_\varepsilon + 1}) - V_{N_\varepsilon + 1} \|_{\sup}\nonumber\\
    &~= \| \mathcal{T}_{\pi_\varepsilon}(V_{\pi_\varepsilon}) - \mathcal{T}_{\pi_\varepsilon}(V_{N_\varepsilon + 1}) \|_{\sup}\nonumber\\
    &~\qquad\qquad + \| \mathcal{T}(V_{N_\varepsilon + 1}) - \mathcal{T}(V_{N_\varepsilon}) \|_{\sup}\nonumber\\
    &~\leq \gamma \| V_{\pi_\varepsilon} - V_{N_\varepsilon + 1} \|_{\sup} + \gamma \| V_{N_\varepsilon + 1} - V_{N_\varepsilon} \|_{\sup},
\end{align}
which implies:
\begin{align}
    &\| V_{\pi_\varepsilon} - V_{N_\varepsilon + 1} \|_{\sup} \leq \frac{\gamma}{1 - \gamma} \| V_{N_\varepsilon + 1} - V_{N_\varepsilon} \|_{\sup}\nonumber\\
    &~\qquad\leq \frac{\gamma}{1 - \gamma} \| V_{N_\varepsilon + 1} - g^* \|_{\sup} + \frac{\gamma}{1 - \gamma} \| g^* - V_{N_\varepsilon} \|_{\sup}\nonumber\\
    &~\qquad< 2\frac{\gamma}{1 - \gamma}\frac{1 - \gamma}{1 + \gamma} \varepsilon = \frac{2\gamma}{1 + \gamma}\varepsilon
\end{align}
Finally, we have:
\begin{align}
    &\| V_{\pi_\varepsilon} - g^* \|_{\sup} \leq \| V_{\pi_\varepsilon} - V_{N_\varepsilon + 1} \|_{\sup} + \| V_{N_\varepsilon + 1} - g^* \|_{\sup}\nonumber\\
    &~\qquad\qquad  < \frac{2\gamma}{1 + \gamma}\varepsilon + \frac{1 - \gamma}{1 + \gamma}\varepsilon = \varepsilon.
\end{align}
This means:
\begin{equation}
    g^*(s) < V_{\pi_\varepsilon}(s) + \varepsilon \leq \left( \sup_{\pi \in \Pi} V_\pi(s) \right) + \varepsilon = V^*(s) + \varepsilon,
\end{equation}
for all states $s \in \mathcal{S}$. Since $\varepsilon$ was arbitrary, we have $g^*(s) \leq V^*(s)$ for all states $s \in \mathcal{S}$.

Since $V^* = g^*$, we conclude that $V^*$ is bounded and upper semicontinuous. By Lemma \ref{lem:usc-measurable-selection}, there is a policy $\pi^* \in \Pi$ satisfying:
\begin{align}
    V_{\pi^*}(s) &= [\mathcal{T}_{\pi^*}(V_{\pi^*})](s)\nonumber\\
    &= r(s, \pi^*(s)) + \gamma V^*(f(s, \pi^*(s)))\nonumber\\ &= \sup_{a \in C(s)} \{ r(s, a) + \gamma V^*(f(s, a)) \}\nonumber\\
    &= [\mathcal{T}(V^*)](s) = V^*(s),
\end{align}
for all states $s \in \mathcal{S}$. This means $\pi^*$ is an optimal policy. Moreover, since $g^*$ is the limit of value iteration from any bounded and upper semicontinuous initial guess, $V^*$ is as well.
\end{proof}

Finally, we conclude with a proof of Lemma \ref{lem:relax-assumption}.

\begin{proof}[Proof of Lemma \ref{lem:relax-assumption}]
Even without the assumption that $\mathcal{S}$ and $\mathcal{A}$ are Polish spaces, $\varphi_{s, a}$ is continuous and bounded, so $\varphi_{s, a} \in \mathcal{C}_b^u(\mathcal{S})$. Therefore, by Lemma \ref{lem:usc-measurable-selection}, there is a policy $\pi_{s, a} \in \Pi$ satisfying \eqref{eq:distance-minimizer}, with $\pi_{s, a}(s) = a$.
\end{proof}

\end{document}